\documentclass{amsart}
\usepackage{amsfonts}
\usepackage{tikz}
\usepackage{tkz-graph}
\usepackage[shortlabels]{enumitem}

\allowdisplaybreaks
\theoremstyle{plain}
\newtheorem{theorem}{Theorem}[section]

\newtheorem{corollary}[theorem]{Corollary}

\newtheorem{definition}[theorem]{Definition}
\newtheorem{example}[theorem]{Example}

\newtheorem{lemma}[theorem]{Lemma}
\newtheorem{notation}[theorem]{Notation}

\newtheorem{proposition}[theorem]{Proposition}
\newtheorem{remark}[theorem]{Remark}

\numberwithin{equation}{section}
\numberwithin{figure}{section}
\numberwithin{table}{section}

\renewcommand{\sp}{\mathrm{sp}}

\begin{document}
\title[$L^{q}$-spectrum for finite type measures]{The $L^{q}$-spectrum for a
class of self-similar measures with overlap}
\author{Kathryn E. Hare}
\address{Dept. of Pure Mathematics\\
University of Waterloo\\
Waterloo, Ont., Canada}
\email{kehare@uwaterloo.ca}
\thanks{This research was supported in part by NSERC grants RGPIN 2016-03719
and 2019-03930}
\author{Kevin G. Hare}
\address{Dept. of Pure Mathematics\\
University of Waterloo\\
Waterloo, Ont., Canada}
\email{kghare@uwaterloo.ca}
\author{Wanchun Shen}
\address{Dept. of Pure Mathematics\\
University of Waterloo\\
Waterloo, Ont., Canada}
\email{w35shen@edu.uwaterloo.ca}
\thanks{This paper is in final form and no version of it will be submitted
for publication elsewhere.}

\begin{abstract}
It is known that the heuristic principle, referred to as the multifractal
formalism, need not hold for self-similar measures with overlap, such as the 
$3$-fold convolution of the Cantor measure and certain Bernoulli
convolutions. In this paper we study an important function in the
multifractal theory, the $L^{q}$-spectrum, $\tau (q)$, for measures of
finite type, a class of self-similar measures that includes these examples.
Corresponding to each measure, we introduce finitely many variants on the $%
L^{q}$-spectrum which arise naturally from the finite type structure and are
often easier to understand than $\tau $. We show that $\tau$ is always
bounded by the minimum of these variants and is equal to the minimum variant
for $q\geq 0$. This particular variant coincides with the $L^{q}$-spectrum
of the measure $\mu$ restricted to appropriate subsets of its support. If
the IFS satisfies particular structural properties, which do hold for the
above examples, then $\tau$ is shown to be the minimum of these variants for
all $q$. Under certain assumptions on the local dimensions of $\mu$, we
prove that the minimum variant for $q \ll 0$ coincides with the straight
line having slope equal to the maximum local dimension of $\mu $. Again,
this is the case with the examples above. More generally, bounds are given
for $\tau$ and its variants in terms of notions closely related to the local
dimensions of $\mu $.
\end{abstract}

\maketitle

\section{\protect\bigskip Introduction}

By the local dimension of a probability measure $\mu$ at a point $x$ in its
support we mean the quantity 
\begin{equation*}
\dim_{\mathrm{loc}}\mu (x)=\lim_{r\rightarrow 0}\frac{\log \mu (B(x,r))}{%
\log r}.
\end{equation*}
It is natural to ask about the values that are attained as local dimensions
of the measure and the size of the sets $E_{\alpha} =\{x:\dim _{\mathrm{loc}%
}\mu (x)=\alpha \}$. For self-similar measures that satisfy the open set
condition, it is well known that the set of attainable local dimensions is a
closed interval and there are simple formulas for the endpoints of this
interval. The Hausdorff dimension of $E_{\alpha}$ is equal to the Legendre
transform of the $L^{q}$-spectrum of $\mu $, $\tau (q)$, at $\alpha$ (see
Definition \ref{D:tau}) meaning, $\dim E_{\alpha}= \inf_{q\in \mathbb{R}%
}(q\alpha -\tau (q))$. In this case $\tau$ is a differentiable function on
all of $\mathbb{R}$. This is known as the multifractal formalism. We refer
the reader to \cite{Fa} for proofs of these facts.

The local behaviour of `overlapping' self-similar measures is not as well
understood and the multifractal formalism need not hold. For instance, in 
\cite{HL} Hu and Lau discovered that the set of local dimensions of the $3$%
-fold convolution of the classical middle-third Cantor measure consists of a
closed interval and an isolated point which is its maximum local dimension.
Specifically, this maximum local dimension occurs at the two endpoints of
the support of the measure. This unexpected property was later found to be
true for more general overlapping regular Cantor-like measures; see \cite%
{BHM, HHN, Sh}. Lau and Ngai in \cite{LN} discovered that if the
self-similar measure arises from an IFS that satisfies only the weak
separation property and its $L^{q}$-spectrum is differentiable at $q_{0}>0$,
then $\dim E_{\alpha}$ is the Legendre transform of $\tau$ at $\alpha$ if $%
\alpha =\tau ^{\prime}(q_{0})$. But there is no guarantee in their theorem
that $\tau$ is differentiable and no information about $\tau (q)$ for $q<0$.
In \cite{LW}, Lau and Wang showed that the $L^{q}$-spectrum of the $3$-fold
convolution of the Cantor measure has one point of non-differentiability at $%
q_{0}<0$. The $L^{q}$-spectrum is equal to the line with slope equal to its
maximum local dimension for $q<q_{0}$ and coincides with the $L^{q}$%
-spectrum of the measure restricted to a closed subinterval of the interior
of its support for $q\geq q_{0}$. A similar result was found for more
general regular, Cantor-like measures in \cite{Sh}. Another much studied
class of overlapping measures are the uniform Bernoulli convolutions with
contraction factor the inverse of a Pisot number $\varrho $. In the case of $%
\varrho =$ golden mean, Feng \cite{FengLimited} found that even though the
set of local dimensions of this measure is a closed interval, its $L^{q}$%
-spectrum is not differentiable at some $q_{0}<0$ and coincides with the
line with slope equal to the maximum local dimension for $q<q_{0}$.

These overlapping measures are all examples of self-similar measures of
finite type. The notion of finite type was introduced by Ngai and Wang in 
\cite{NW}, and is essentially weaker than the open set condition and
stronger than the weak separation property. In \cite{FengSmoothness}, Feng
proved that the $L^{q}$-spectrum of any finite type measure is
differentiable for all $q>0$. In \cite{FengLimited, FengLy}, he extensively
studied the local dimension theory of uniform Bernoulli convolutions with
simple Pisot inverses as the contraction factor. The local dimension theory
was extended to general self-similar measures of finite type in a series of
papers by two of the authors with various coauthors, \cite{HHM, HHN, HHS}.

In this note, we continue the study of the $L^{q}$-spectrum of measures of
finite type. Important structural building blocks for measures of finite
type are combinatorial objects known as loop classes. Our general strategy
is to `decompose' the support of the measure into these loop classes, study
the restriction of its $L^{q}$-spectrum to each such a loop class $L$
(denoted $\tau _{L})$ and then try to recover information about $\tau$ from
the data collected for each loop class. As seen in \cite{HHM, HHN, HHS}, the
set of local dimensions corresponding to a given loop class is often a
closed interval and hence according to the multifractal philosophy, it is
reasonable to expect $\tau _{L}$ to be better behaved than $\tau $.

We have that $\tau \leq \min \tau _{L}$ where the minimum is taken over all
maximal loop classes $L$. We prove that equality holds under certain
structural assumptions. In particular, we give criteria which ensures there
is a some $q_{0}<0$ such that $\tau (q)=qd$ for $d=\max \{\dim _{\mathrm{loc}%
}\mu (x):x\}$ and all $q<q_{0}$. 
The line $y=qd$ arises because it is the
function $\tau _{L}(q)$ for a (suitable) singleton maximal loop class $L$.
In Example \ref{ExSpecial} we give an example showing that if these 
    criteria are not met, it is possible for $\tau$ not to have this property.
In particular, this example has the property that $\tau(q)/q \not\to d$ as
    $q \to -\infty$.

In general, for $q\geq q_{0}$, 
    the $L^{q}$-spectrum coincides with $\tau _{L}$ for $L$
the special maximal loop class known as the essential class. It also
coincides with the $L^{q}$-spectrum of the measure restricted to various
proper subsets of its support which are easily described in terms of the
finite type data (and which in many examples can be any proper closed subset
of the interior of its support). A lower bound for $q_{0}$ is given in terms
of the finite type data. These structural assumptions are satisfied in many
examples, including regular Cantor-like measures and biased Bernoulli
convolutions with contraction factor the inverse of a simple Pisot number.

The proofs of most of these results, including a detailed discussion of what
they imply about finite type Cantor-like measures and Bernoulli
convolutions, can be found in Section \ref{S:Main}.

In Section \ref{S:LoopCl} we discuss the local dimensional behaviour of the
measure on loop classes and compare this with the (usual) local dimension of
the measure. In Section \ref{sp} we introduce the notion of the loop class $%
L^{q}$-spectrum, $\tau _{L}$, and establish basic properties of these
variants on the $L^{q}$-spectrum. In particular, we find bounds on the
functions $\tau _{L}$, determine their asymptotic behaviour and prove that
the $L^{q}$-spectrum $\tau$ is dominated by the minimum of these $\tau _{L}$%
. In Section \ref{sec:basic}, we recall the definitions, notation and basic
facts about finite type measures that are needed in the paper.

In \cite{FL} (see also \cite{FengLy}), Feng and Lau showed that suitable
restrictions of self-similar measures having only the weak separation
property satisfy the multifractal formalism. However, the local dimensions
of the original measure and its restrictions will not, in general, coincide
at all points in the support of the restricted measure. Moreover, in many
examples, the $L^{q}$-spectrum of the original measure and its restrictions
only agree for large $q$.

\section{Basic Definitions and Terminology}

\label{sec:basic}

We begin by reviewing the notion of finite type and the related concepts and
terminology that will be used throughout the paper. These notes are
basically summarized from \cite{FengSmoothness, HHM, HHS} where the facts
which are stated here are either proved or references given.

\subsection{Iterated function systems and finite type}

By an iterated function system (IFS) we will mean a finite set of
contractions, 
\begin{equation}
S_{j}(x)=\rho _{j}x+d_{j}:\mathbb{R\rightarrow R}\text{ for }j=0,1,\dots,m,
\label{IFS}
\end{equation}
where $m\geq 1$ and $0<$ $\left\vert \rho _{j}\right\vert <1$. When all $%
\rho _{j}$ are equal and positive, the IFS is referred to as \textbf{%
equicontractive}. Each IFS generates a unique invariant, compact set $K$,
known as its associated self-similar set, satisfying $K=
\bigcup_{j=0}^{m}S_{j}(K)$. By rescaling the $d_{j}$, if necessary, we can
assume the convex hull of $K$ is $[0,1]$.

Assume we are given probabilities $p_{j}>0$ satisfying $%
\sum_{j=0}^{m}p_{j}=1 $. There is a unique self-similar probability measure $%
\mu$ associated with the IFS $\{S_{j}\}_{j=0}^{m}$ and probabilities $%
\{p_{j}\}_{j=0}^{m}$, supported on the self-similar set associated with the
IFS and satisfying the rule 
\begin{equation*}
\mu =\sum_{j=0}^{m}p_{j}\mu \circ S_{j}^{-1}.
\end{equation*}
The measure is said to be equicontractive if the IFS is equicontractive.

Given a finite word $\omega =(\omega _{1},\dots,\omega _{j})$, on the
alphabet $\{0,1,\dots,m\}$, we will let $\omega ^{-}=(\omega
_{1},\dots,\omega _{j-1})$, $S_{\omega}=S_{\omega _{1}}\circ S_{\omega
_{2}}\circ \cdot \cdot \cdot \circ S_{\omega _{j}}$ and $\rho
_{\omega}=\prod_{i=1}^{j}\rho _{\omega _{i}}$. Let 
\begin{equation*}
\rho _{\min}=\min_{j}\left\vert \rho _{j}\right\vert
\end{equation*}
and put 
\begin{equation*}
\Lambda _{n}=\{\text{finite words }\omega :\left\vert \rho _{\omega
}\right\vert \leq \rho _{\min}^{n}\text{ and }\left\vert \rho _{\omega
^{-}}\right\vert >\rho _{\min}^{n}\}.
\end{equation*}

The notion of finite type was introduced by Ngai and Wang in \cite{NW}. The
definition we will use is slightly less general, but is simpler and seen to
be equivalent to the finite type definition given for equicontractive IFS in 
\cite{FengSmoothness}. It includes all the examples of finite type measures
in $\mathbb{R}$ of which we are aware.

\begin{definition}
Assume $\{S_{j}\}$ is an IFS as in equation \eqref{IFS}. The words $\omega
,\tau \in \Lambda _{n}$ are said to be neighbours if $S_{\omega}(0,1)\cap
S_{\tau}(0,1)\neq \emptyset $. Denote by $\mathcal{N}(\omega )$ the set of
all neighbours of $\omega $. We say that $\omega \in \Lambda _{n}$ and $\tau
\in \Lambda _{m}$ have the same neighbourhood type if there is a map $%
f(x)=\pm \rho _{\min}^{n-m}x+c$ such that 
\begin{equation*}
\{f\circ S_{\eta}:\eta \in \mathcal{N}(\omega )\}=\{S_{\nu}:\nu \in \mathcal{%
N}(\tau )\}\ \ \mathrm{and}\ \ f\circ S_{\omega}=S_{\tau}.
\end{equation*}
The IFS is said to be of \textbf{finite type} if there are only finitely
many neighbourhood types. Any associated self-similar measure is also said
to be of finite type.
\end{definition}

It was shown in \cite{Ng} that an IFS of finite type satisfies the weak
separation property, but not necessarily the open set condition.

Here are two interesting and much studied classes of measures of finite type
that fail to satisfy the open set condition. We will refer to these often in
the paper.

\begin{example}[Bernoulli convolutions]
Consider the IFS: $S_{0}(x)=\varrho x, S_{1}(x)=\varrho x+1-\varrho$, where $%
1 < \varrho < 2$ is the inverse of a Pisot number\footnote{%
Recall that a Pisot number is an algebraic integer greater than $1$, all of
whose Galois conjugates are $<1$ in absolute value.} such as the golden
mean. If the two probabilities are equal ($p_{0}=p_{1}=1/2$) we call the
associated self-similar measure a uniform Bernoulli convolution and
otherwise it is said to be biased. The self-similar set is $[0,1]$. These
measures are all of finite type.
\end{example}

\begin{example}[$(m,d$)-Cantor measures]
Consider the IFS: 
\begin{equation*}
S_{j}(x)=\frac{x}{d}+\frac{j(d-1)}{dm}\text{ for }j=0,\dots,m
\end{equation*}
with integers $m\geq d\geq 2$ and probabilities $\{p_{j}\}$. The associated
self-similar measures are called $(m,d)$-Cantor measures and they have
support $[0,1]$. If we take $p_{j}=\binom{m}{j}2^{-m}$, the resulting
measure is the $m$-fold convolution of the uniform Cantor measure on the
Cantor set with ratio of dissection $1/d$, rescaled to have support $[0,1]$.
For example, if $d=3=m$ and the probabilities are $1/8,3/8,3/8,1/8$, the
self-similar measure is the rescaled $3$-fold convolution of the classical
middle-third Cantor measure. These measures are all of finite type.
\end{example}

\subsection{Net intervals and characteristic vectors}

\begin{definition}
For each positive integer $n$, let $h_{1},\dots ,h_{s_{n}}$ be the
collection of elements of the set $\{S_{\omega }(0),S_{\omega }(1):\omega
\in \Lambda _{n}\}$, listed in increasing order. Put 
\begin{equation*}
\mathcal{F}_{n}=\{[h_{j},h_{j+1}]:1\leq j\leq s_{n}-1\text{ and }%
(h_{j},h_{j+1})\cap K\neq \emptyset \}\text{.}
\end{equation*}%
Elements of $\mathcal{F}_{n}$ are called \textbf{net intervals of level }$n$%
. The interval $[0,1]$ is understood to be the (only) net interval of level $%
0$.
\end{definition}

For each IFS\ of finite type there is some $c>0$ such that 
\begin{equation*}
c\rho _{\min}^{n}\leq \ell(\Delta) \leq \rho _{\min}^{n}\text{.}
\end{equation*}
for all net intervals $\Delta$ of level $n$. Here $\ell(\Delta)$ is the
length of a net interval $\Delta$.

Given $\Delta =[a,b]\in \mathcal{F}_{n}$, we let $\ell _{n}(\Delta )=\rho
_{\min }^{-n}(b-a)$. By the \textbf{neighbour set} of $\Delta $ we mean the
ordered tuple

\begin{equation}
V_{n}(\Delta )=((a_{1},L_{1}),(a_{2},L_{2}),\dots ,(a_{jJ},L_{J})),
\label{nb}
\end{equation}
where for each $i$ there is some $\sigma _{i}\in \Lambda _{n}$ such that $%
S_{\sigma _{i}}(K)\cap \mathrm{int}\Delta \neq \emptyset$, $\rho
_{\min}^{-n}\rho _{\sigma _{i}}=L_{i}$ and $\rho _{\min}^{-n}(a-S_{\sigma
_{i}}(0))=a_{i}$. (We will say $\sigma _{i}$ \textbf{is associated with} $%
(a_{i}, L_{i})$ \textbf{and} $\Delta $.) We will order these tuples so that $%
a_{i}\leq a_{i+1}$ and if $a_{i}=a_{i+1}$, then $L_{i}<L_{i+1}$. We note
that some $L_{i}$ could be negative if there are negative contraction
factors.

For each $\Delta \in \mathcal{F}_{n}$, $n\geq 1$, there is a unique element $%
\widehat{\Delta }\in \mathcal{F}_{n-1}$ which contains $\Delta $, called the 
\textbf{parent} of \textbf{child} $\Delta $. Suppose $\Delta \in \mathcal{F}%
_{n}$ has parent $\widehat{\Delta }$. If $\widehat{\Delta }$ has $J$
children with the same normalized length and neighbourhood set as $\Delta $,
we order these from left to right as $\Delta _{1},\Delta _{2},\dots ,\Delta
_{J}$ and denote by $t_{n}(\Delta )$ the integer $t$ such that $\Delta
_{t}=\Delta $.

\begin{definition}
The \textbf{characteristic vector of} $\Delta \in \mathcal{F}_{n}$ is
defined to be the triple 
\begin{equation*}
\mathcal{C}_{n}(\Delta _{n})=(\ell _{n}(\Delta ),V_{n}(\Delta ),t_{n}(\Delta
)).
\end{equation*}
We denote by $\gamma _{0}$ the characteristic vector of $[0,1]$.
\end{definition}

Each net interval $\Delta \in \mathcal{F}_{n}$ is uniquely identified by the 
$(n+1)$-tuple $(\gamma _{0},\gamma _{1},\dots,\gamma _{n})$, called the 
\textbf{symbolic representation of} $\Delta $, where $\gamma _{j}=$ $%
\mathcal{C}_{j}(\Delta _{j})$, $\Delta _{0}=[0,1]$, $\Delta _{n}=\Delta $,
and for each $j=1,\dots,n$, $\Delta _{j-1}$ is the parent of $\Delta _{j}$.
Similarly, for each $x\in \lbrack 0,1]$, the \textbf{symbolic representation
of} $x$ will be the (infinite) sequence of characteristic vectors $(\mathcal{%
C}_{0}(\Delta _{0}),\mathcal{C}_{1}(\Delta _{1}),\dots)$ where $x\in \Delta
_{n}\in \mathcal{F}_{n}$ for each $n$ and $\Delta _{j-1}$ is the parent of $%
\Delta _{j}$. The symbolic representation uniquely determines $x$ and is
unique unless $x$ is the endpoint of some net interval, in which case there
can be two different symbolic representations (and two net intervals of
level $n$ containing $x$). We will write $\Delta _{n}(x)$ for any net
interval of level $n$ containing $x$. An important fact is that the
characteristic vector of a child is uniquely determined by the
characteristic vector of the parent, thus we can also speak of the
parent/child of characteristic vectors.

By an \textbf{admissible path} $(\chi _{1},\chi _{2},\dots)$ (or path, for
short) we mean a (finite or infinite) sequence of characteristic vectors
where each $\chi _{j+1}$ is the characteristic vector of a child of $\chi
_{j}$. Note that a path need not start with $\gamma _{0}$. We write $%
\left\vert \eta \right\vert$ for the length of the finite path $\eta $. We
remark that if we write $(\chi _{1},\chi _{2},\dots)$ (finite or infinite)
for some characteristic vectors $\chi _{i}$, it is implied that this is an
admissible path. When we write $\sigma |n$ we mean the restriction of the
path $\sigma$ to its first $n+1$ letters.

Since the characteristic vectors of the children of $\Delta$ depend only on
the characteristic vector of $\Delta $, we can construct a finite directed
graph of characteristic vectors, called the transition graph, where we have
a directed edge from $\gamma$ to $\beta$ if there is a net interval $\Delta$
with characteristic vector $\gamma$ and a child of $\Delta$ with
characteristic vector $\beta $. The paths in the graph are the admissible
paths.

\begin{example}
In the case of the Bernoulli convolution with contraction factor the inverse
of the golden mean, Feng \cite[Section 4.1]{FengLimited} has shown that
there are 7 characteristic vectors, $\gamma _{0},\gamma _{1},\dots,\gamma
_{6}$. See Figure \ref{fig:Bernoulli} for its transition graph. The symbolic
representation of $0$ is the sequence $(\gamma _{0},\gamma _{1},\gamma
_{1},\dots)$, and $(\gamma _{0},\gamma _{3},\gamma _{3},\dots)$ is the
symbolic representation of $1$.
\end{example}

Every IFS of finite type has only finitely many (distinct) characteristic
vectors; see \cite{HHS}. We will denote this finite set of characteristic
vectors by $\Omega $.

Suppose $\Delta$ and $\Delta ^{\prime}$ are any net intervals at levels $n$
and $m$ respectively, with the same characteristic vector. Assume their
(common) neighbour set is $V=((a_{j},L_{j}))_{j=1}^{J}$ . It is a
consequence of the definitions that there is a constant $d$ such that if $%
\sigma _{i}\in \Lambda _{n}$ is associated with $(a_{i},L_{i})$ and $\Delta$%
, and $\sigma _{i}^{\prime}\in \Lambda _{m}$ is associated with $%
(a_{i},L_{i})$, but with respect to $\Delta ^{\prime}$, then $S_{\sigma
_{i}^{\prime}}=\rho _{\min}^{m-n}S_{\sigma _{i}}+d$.

As $S_{\sigma _{i}}(K)\cap \mathrm{int}\Delta$ is non-empty, there is some
word $\lambda _{i}\in \Lambda _{N_{i}}$ such that $S_{\sigma _{i}\lambda
_{i}}[0,1]\subseteq \Delta $. But then, also, $S_{\sigma _{i}^{\prime
}\lambda _{i}}[0,1]\subseteq \Delta ^{\prime}$. Since there are only
finitely many characteristic vectors and only finitely many paths of bounded
length, this proves Part \ref{incl i} below. Part \ref{incl ii} is proved
similarly.

\begin{lemma}
\label{incl} There is a finite set of words $\mathcal{W}$ and $N\in \mathbb{N%
}$ with the following properties:

\begin{enumerate}
\item If $\Delta_{n}\in \mathcal{F}_{n}$ and $S_{\sigma}[0,1]\supseteq
\Delta _{n}$ for some $\sigma \in \Lambda _{n}$, then there is some $\nu \in 
\mathcal{W}$ such that $\sigma \nu \in \Lambda _{n+N}$ and $S_{\sigma \nu
}[0,1]\subseteq \Delta _{n}$. \label{incl i}

\item If $\Delta _{n+N}\subseteq \Delta _{n}$ belong to $\mathcal{F}_{n+N}$
and $\mathcal{F}_{n}$ respectively, and $S_{\sigma}[0,1]\supseteq \Delta
_{n} $ for $\sigma \in \Lambda _{n}$, then there is some $v\in \mathcal{W}$
such that $\sigma v\in \Lambda _{n+N}$ and $S_{\sigma v}[0,1]\bigcap \Delta
_{n+N} $ is empty. \label{incl ii}
\end{enumerate}
\end{lemma}

\subsection{Local dimensions and transition matrices}

\label{sec:local}

\begin{definition}
Given a probability measure $\mu$ on $\mathbb{R}$, by the \textbf{lower
local dimension} of $\mu$ at $x\in \mathrm{supp} \mu$, we mean the number 
\begin{equation*}
\underline{\dim}_{\mathrm{loc}}\mu (x)=\liminf_{r\rightarrow 0^{+}}\frac{%
\log \mu (x-r,x+r)}{\log r}.
\end{equation*}
Replacing the $\liminf$ by $\limsup$ gives the upper local dimension, $%
\overline{\dim}_{\mathrm{loc}}\mu (x)$, and if these two are equal, the
common value is the local dimension of $\mu$ at $x$, denoted $\dim _{\mathrm{%
loc}}\mu (x)$.
\end{definition}

If $\mu$ is a measure of finite type, the local dimensions of $\mu$ can be
expressed in terms of measures of net intervals. Indeed, as all net
intervals of level $n$ have lengths comparable to $\rho _{\min}^{n}$, it
follows that 
\begin{eqnarray}
\dim _{\mathrm{loc}}\mu (x) &=&\lim_{n\rightarrow \infty}\frac{\log (\mu
(\Delta _{n}(x))+\mu (\Delta _{n}^{+}(x))+\mu (\Delta _{n}^{-}(x)))}{n\log
\rho _{\min}}  \label{locdimft} \\
&\leq &\lim_{n\rightarrow \infty}\frac{\log (\mu (\Delta _{n}(x))}{n\log
\rho _{\min}},  \notag
\end{eqnarray}
where $\Delta _{n}^{+}(x),\Delta _{n}^{-}(x)$ are the adjacent, $n$'th level
net intervals on each side of $\Delta _{n}(x)$. A similar statement holds
for the upper and lower local dimensions.

\begin{definition}
Let $\mu $ be a measure of finite type with the notation as in the previous
subsection. Let $\Delta =[a,b]\in \mathcal{F}_{n}$ and $\widehat{\Delta }%
=[c,d]\in \mathcal{F}_{n-1}$ be its parent. Assume $V_{n}(\Delta
)=((a_{j},L_{j}))_{j=1}^{J}$ and $V_{n-1}(\widehat{\Delta }%
)=((c_{i},M_{i}))_{i=1}^{I}$. The \textbf{primitive transition matrix}, $T(%
\mathcal{C}_{n-1}(\widehat{\Delta }),\mathcal{C}_{n}(\Delta ))$, is the $%
I\times J$ matrix $(T_{ij})$ which encapuslates information about the
relationship between the $(c_{i},M_{i})\in V_{n-1}(\widehat{\Delta })$ and $%
(a_{j},L_{j})\in V_{n}(\Delta )$. To be precise, let $\sigma _{i}\in \Lambda
_{n-1}$ be such that $\rho _{\min }^{-n+1}(c-S_{\sigma }(0))=c_{i}$ and $%
\rho _{\min }^{-n+1}\rho _{\sigma }=M_{i}$. Let $\mathcal{T}_{i,j}$ be the
set all $\omega $ such that $\sigma \omega \in \Lambda _{n}$, $\rho _{\min
}^{-n}(a-S_{\sigma _{i}\omega }(0))=a_{j}$ and $\rho _{\min }^{-n}\rho
_{\sigma _{i}\omega }=L_{j}$. Notice that $\mathcal{T}_{i,j}$ depends only $%
S_{\sigma _{i}}$ (or equivalently on $c_{i}$ and $M_{i}$), and not on the
choice of $\sigma _{i}$. We define $T_{i,j}=\sum_{\omega \in \mathcal{T}%
_{i,j}}p_{\omega }$ where the empty sum is taken to be $0$. Given a path $%
(\gamma _{1},\dots ,\gamma _{n})$, we let 
\begin{equation*}
T(\gamma _{1},\dots ,\gamma _{n})=T(\gamma _{1},\gamma _{2})T(\gamma
_{2},\gamma _{3})\cdot \cdot \cdot T(\gamma _{n-1},\gamma _{n}).
\end{equation*}%
We call any such product a \textbf{transition matrix}.
\end{definition}

As explained in \cite{FengSmoothness, HHS}, there are positive constants $%
c_{1},c_{2}$ such that whenever $\Delta \in \mathcal{F}_{n}$ has symbolic
representation $(\gamma _{0},\gamma _{1},\dots,\gamma _{n})$, then 
\begin{equation*}
c_{1}\mu (\Delta _{n})\leq \left\Vert T(\gamma _{0},\gamma _{1},\dots,\gamma
_{n})\right\Vert \leq c_{2}\mu (\Delta _{n}).
\end{equation*}
We say that $\mu (\Delta _{n})$ and $\left\Vert T(\gamma _{0},\gamma
_{1},\dots,\gamma _{n})\right\Vert$ are \textbf{comparable}.

An important fact about transition matrices is that each column of any
transition matrix contains a non-zero entry. Here are two useful
consequences of this. The proofs are left as an exercise and follow from the
definition of the matrix norm. Note that we say a matrix is \textbf{positive}
if all its entries are strictly positive.

\begin{lemma}
\label{useful}Let $A, B, C$ be transition matrices with $B$ positive.

\begin{enumerate}
\item There are positive constants $a_{1},a_{2}$, depending only on $A$,
such that 
\begin{equation*}
a_{1}\left\Vert C\right\Vert \leq \left\Vert AC\right\Vert \leq \left\Vert
A\right\Vert \left\Vert C\right\Vert \leq a_{2}\left\Vert C\right\Vert .
\end{equation*}
\label{useful i}

\item There is a constant $b=b(B)>0$ (independent of $A,C)$ such that 
\begin{equation*}
\left\Vert ABC\right\Vert \geq b\left\Vert A\right\Vert \left\Vert
C\right\Vert .
\end{equation*}
\label{useful ii}
\end{enumerate}
\end{lemma}

\section{Local dimensional behaviour on Loop classes}

\label{S:LoopCl}

From here on, unless we say otherwise $\mu$ will be a self-similar measure
arising from an IFS of finite type, with the notation as in the previous
section.

\subsection{Loop classes and the Essential class}

A non-empty subset $L$ of the set of characteristic vectors $\Omega$ is
called a \textbf{loop class} if whenever $\alpha ,\beta \in L$, then there
is an admissible path $(\gamma _{1},\dots,\gamma _{J})$ of characteristic
vectors $\gamma _{i}\in L$ such that $\alpha =\gamma _{1}$ and $\beta
=\gamma _{J}$.

A loop class $L$ is called an \textbf{essential class} if, in addition,
whenever $\alpha \in L$ and $\beta \in \Omega$ is a child of $\alpha $, then 
$\beta \in L$. It was shown in \cite{FengLy, HHS} that there is always a
unique essential class which we will denote by $E$. A loop class is called 
\textbf{maximal} if it is not properly contained in any other loop class.

Given a set of characteristic vectors $L$ containing a loop class, we will
say that the infinite path $\sigma$ belongs to $L_{e}$, and write $\sigma
\in L_{e}$, if $\sigma =(\sigma _{1},\sigma _{2},\dots)$ with $\sigma
_{j}\in L$ eventually, i.e., there exists an index $j_{0}$ such that $\gamma
_{j}\in L$ for all $j\geq j_{0}$. Of course, every infinite path will belong
to some maximal loop class $L$ eventually. We will say that the (finite or
infinite) path $\sigma$ $\in L_{a}$ if all the letters of $\sigma$ belong to 
$L$. If the path $\sigma$ begins with $\gamma _{0}$, the characteristic
vector of $[0,1]$, we will write $\sigma \in L_{e}^{0}$ or $L_{a}^{0}$,
appropriately.

If $x$ has a symbolic representation $\sigma \in L_{e}^{0}$, we will say
that $x\in K_{L}$. Of course, $x$ can belong to both $K_{L_{1}}$ and $%
K_{L_{2}}$ for different maximal loop classes $L_{1},L_{2}$ only if $x$ is a
boundary point of a net interval with symbolic representations coming from
both $L_{1}$ and $L_{2}$. When $x\in K_{E}$ we will say $x$ is an \textbf{%
essential point}.

\begin{example}
For both the Bernoulli convolution with contraction factor the inverse of
the golden mean and the $(m,d)$-Cantor measures with $m\geq d$, it is the
case that the set of essential points is $(0,1)$. For more details see
Examples \ref{Cantor} and \ref{Pisot}.
\end{example}

\subsection{Local dimensional behaviour of paths}

Motivated by the notion of the local dimension at a point, we introduce a
related notion for paths.

\begin{notation}
Given any infinite path $\sigma$, put 
\begin{equation*}
\underline{d}(\sigma )=\liminf_{n\rightarrow \infty}\frac{\log \left\Vert
T(\sigma |n)\right\Vert}{n\log \rho _{\min}}\text{ and }\overline{d} (\sigma
)=\limsup_{n\rightarrow \infty}\frac{\log \left\Vert T(\sigma |n)\right\Vert%
}{n\log \rho _{\min}}.
\end{equation*}
We write $d(\sigma )$ if $\underline{d}(\sigma )=\overline{d}(\sigma )$.

Let $L$ be any non-empty subset of $\Omega$ containing a loop class and set 
\begin{equation*}
d_{\min}^{L}=\inf_{\sigma \in L_{e}^{0}}\underline{d}(\sigma )\text{ and }
d_{\max}^{L}=\sup_{\sigma \in L_{e}^{0}}\text{ }\overline{d}(\sigma ).
\end{equation*}
\end{notation}

Since $\left\Vert T(\eta ,\sigma )\right\Vert$ and $\left\Vert T(\sigma
)\right\Vert$ are comparable whenever $\eta ,\sigma$ are finite paths, with
constants of comparability depending only on $\eta$ (see Lemma \ref{useful}%
), $\underline{d}(\sigma )=\underline{d}(\sigma ^{\prime})$ where $\sigma
^{\prime}$ omits the initial segment of $\sigma$ that contains the letters
not in $L$. A similar statement holds for $\overline{d}$. Thus 
\begin{equation*}
d_{\min}^{L}=\inf_{\sigma \in L_{a}}\underline{d}(\sigma )\text{ and }
d_{\max}^{L}=\sup_{\sigma \in L_{a}}\overline{d}(\sigma ).
\end{equation*}

\begin{lemma}
We have 
\begin{equation*}
0<d_{\min}^{\Omega}=\min_{L}d_{\min}^{L}\text{ and }d_{\max}^{\Omega
}=\max_{L}d_{\max}^{L}<\infty
\end{equation*}
where the minimum and maximum are over all maximal loop classes $L$.
\end{lemma}

\begin{proof}
Lemma \ref{incl} \ref{incl ii} implies that there is an index $N$ and a
finite set of words $\mathcal{W}$ such that if $\nu$ is associated with an
element of the neighbour set of $\Delta _{n}(x)$, then there is some $\omega
\in \mathcal{W} $ such that $\nu \omega$ is not associated with any element
of the neighbour set of $\Delta _{n+N}(x)$. Thus, if $\sigma =(\gamma
_{1},\dots,\gamma _{N+1})$ is any path of length $N+1$, then the sum of each
row of the transition matrix $T(\sigma )$ is at most $1-\varepsilon$, where $%
\varepsilon =\min \{p_{\omega}: \omega \in \mathcal{W}\}$. (Think of $\gamma
_{1}$ as the characteristic vector of some $\Delta _{n}(x)$ and $\gamma
_{N+1}$ as the characteristic vector of its descendent $\Delta _{n+N}(x)$).

If we let $\left\Vert T\right\Vert _{r}=\max_{i}\sum_{j}\left\vert
T_{ij}\right\vert $ be the maximum row sum norm, then one can easily verify
that $\left\Vert T_{1}T_{2}\right\Vert _{r}\leq \left\Vert T_{1}\right\Vert
_{r}\left\Vert T_{2}\right\Vert _{r}$. Furthermore, $\left\Vert T\right\Vert
\leq C\left\Vert T\right\Vert _{r}$ where $C$ is a bound on the number of
rows of matrix $T$. If $\sigma =(\gamma _{1},\gamma _{2},\dots )$ is any
infinite path, then we can factor $T(\sigma |n)$ as $T(\eta _{1})\cdot \cdot
\cdot T(\eta _{J})T(\lambda )$ where $\eta _{i}=(\gamma _{(i-1)N+1},\dots
,\gamma _{iN+1})$ are paths of length $N+1$, $J=\left[ \frac{n-1}{N}\right] $
and $\lambda =(\gamma _{JN+1},\dots ,\gamma _{n})$ is a path of length $\leq
N$. As each $\left\Vert T(\eta _{i})\right\Vert _{r}\leq 1-\varepsilon $,
and there are only finitely many paths of length at most $N$, the
submultiplicativity of the $r$-norm implies 
\begin{equation*}
\left\Vert T(\sigma |n)\right\Vert \leq C\left\Vert T(\sigma |n)\right\Vert
_{r}\leq C(1-\varepsilon )^{J}\max_{\left\vert \lambda \right\vert \leq
N}\left\Vert T(\lambda )\right\Vert \leq C^{\prime }(1-\varepsilon )^{\left[ 
\frac{n-1}{N}\right] }
\end{equation*}%
for a suitable constant $C^{\prime }$. Hence 
\begin{equation}
\underline{d}(\sigma )=\liminf_{n\rightarrow \infty }\frac{\log \left\Vert
T(\sigma |n)\right\Vert }{n\log \rho _{\min }}\geq \frac{\log (1-\varepsilon
)}{N\log \rho _{\min }\text{ }}>0\text{ for all }\sigma .  \label{lower}
\end{equation}

On the other hand, if $N$ is the integer of Lemma \ref{incl} \ref{incl i},
then any $\Delta \in \mathcal{F}_{n}$ contains $S_{\nu}[0,1]$ for some $\nu
\in \Lambda _{n+N}$. Hence 
\begin{equation}
\mu (\Delta )\geq (\min p_{j})^{s(n+N)},  \label{munet}
\end{equation}
where $s$ is chosen such that any word in $\Lambda _{k}$ is of length at
most $sk$.

Equivalently, there is a constant $C>0$ such that $\left\Vert T(\sigma
|n)\right\Vert \geq C(\min p_{j})^{s(n+N)}$ for all infinite paths $\sigma$.
Thus 
\begin{equation}
\overline{d}(\sigma )=\limsup_{n\rightarrow \infty}\frac{\log \left\Vert
T(\sigma |n)\right\Vert}{n\log \rho _{\min}}\leq \frac{s\log (\min p_{j})} {%
\log \rho _{\min}}<\infty .  \label{upper}
\end{equation}
The bounds (\ref{lower}) and (\ref{upper}) obviously imply $d_{\min}^{L}$
and $d_{\max}^{L}$ are bounded above and below from $0$. Since every
infinite path $\sigma$ belongs to $L$ eventually for a unique choice of
maximal loop class L the proof is complete.
\end{proof}

\subsection{Relationships between $d_{\min}^{L}, d_{\max}^{L}$ and local
dimensions}

If $x\in K_{L}$ has symbolic representation $\sigma \in L_{e}^{0}$ and $%
\Delta _{n}(x)$ has symbolic representation $\sigma |n$, then the
comparability of $\mu (\Delta _{n})$ and $\left\Vert T(\sigma |n)\right\Vert 
$ when $\sigma |n$ is the symbolic representation of $\Delta _{n}$, together
with (\ref{locdimft}), shows 
\begin{equation}
\underline{d}(\sigma )=\liminf_{n\rightarrow \infty}\frac{\log \mu (\Delta
_{n}(x))}{n\log \rho _{\min}}\geq \underline{\dim}_{\mathrm{loc}}\mu (x).
\label{dmin-locdim}
\end{equation}
Similarly, 
\begin{equation*}
\overline{\dim}_{\mathrm{loc}}\mu (x)\leq \overline{d}(\sigma ).
\end{equation*}
In particular, if $L$ contains a loop class, then 
\begin{equation}
\overline{\dim}_{\mathrm{loc}}\mu (x)\leq d_{\max}^{L}\text{ for all }x\in
K_{L}.  \label{dmax-locdim}
\end{equation}

\begin{definition}
We will say that an infinite path $\sigma$ is a \textbf{periodic path with
period }$\theta$ if $\sigma =(\eta ,\theta ^{-},\theta ^{-},\dots)$ for some
initial finite path $\eta$ and \textbf{cycle} $\theta =(\theta
_{1},\dots,\theta _{k},\theta _{1})$. We call $x$ a \textbf{periodic point}
if it has a periodic symbolic representation.
\end{definition}

An example of a periodic point is the boundary point of a net interval.

\begin{notation}
Denote by $\sp(M)$ the spectral radius of the square matrix $M$.
\end{notation}

\begin{example}
\label{Ex:periodic}If $\sigma$ is a periodic path with period $\theta $,
then 
\begin{equation}
d(\sigma )=\underline{d}(\sigma )=\overline{d}(\sigma )=\lim_{k\rightarrow
\infty}\frac{\log \left\Vert (T(\theta ))^{k}\right\Vert}{\left\vert \theta
^{-}\right\vert k\log \rho _{\min}}=\frac{\log \sp(T(\theta ))} {\left\vert
\theta ^{-}\right\vert \log \rho _{\min}}.  \label{d-cycle}
\end{equation}

Similarly, if $x$ is a periodic point with a unique periodic symbolic
representation $\sigma $, then $\dim _{\mathrm{loc}}\mu (x)=d(\sigma )$. If $%
x$ has two different symbolic representations, $\sigma ,\tau $, then these
are both necessarily periodic and 
\begin{equation*}
\dim _{\mathrm{loc}}\mu (x)=\min (d(\sigma ),d(\tau )).
\end{equation*}
See \cite[Prop. 2.7]{HHN} for details. In Example \ref{ExSp1} we show that
it is possible to have $d(\sigma )\neq d(\tau )$.
\end{example}

An equicontractive self-similar measure is said to be \textbf{regular} if $%
p_{0}=p_{m}=\min p_{j}$ where the $S_{j}$ are ordered so that $%
d_{0}<d_{1}<\cdot \cdot \cdot <d_{m}$. It was shown in \cite[Thm. 3.2]%
{FengLimited} (see also \cite[Cor. 3.7]{HHM}) that if $\mu$ is a regular,
finite type measure, then the $\mu $-measures of adjacent net intervals are
comparable and consequently (\ref{locdimft}) implies 
\begin{equation}
\underline{\dim}_{\mathrm{loc}}\mu (x)=\liminf_{n\rightarrow \infty}\frac{%
\log (\mu (\Delta _{n}(x))}{n\log \rho _{\min}}=\liminf_{n\rightarrow \infty}%
\frac{ \log \left\Vert T(\sigma |n)\right\Vert}{n\log \rho _{\min}}=%
\underline{d} (\sigma )  \label{locdimreg}
\end{equation}
when $x$ has symbolic representation $\sigma $, and similarly for the
(upper) local dimension. Consequently, under the regularity assumption 
\begin{equation}
d_{\min}^{L}=\inf \{\underline{\dim}_{\mathrm{loc}}\mu (x):x\in K_{L}\},%
\text{ }  \label{R1}
\end{equation}
and similarly for $d_{\max}^{L}$ and the upper local dimensions. But without
this assumption, these statements need not be true. Here is one example.

\begin{example}
\label{ExSp1}Consider the IFS with contractions $S_{j}(x)=x/3+d_{j}$ for $%
d_{j}=0,1/9,1/3,1/2,2/3$ and probabilities $p_{j}=4/17$ for $j=0,1,3,4$ and $%
p_{2}=1/17$. This IFS is of finite type and has 19 characteristic vectors.

In particular, $L=\{\gamma _{4}\}$ is a singleton maximal loop class with $%
K_{L}=\{1/2\}$. In this case, the only infinite word in $L_{e}^{0}$ is $%
\sigma :=(\gamma _{0},\gamma _{4},\gamma _{4},\gamma _{4},\dots )$. It can
be checked that 
\begin{equation*}
d(\sigma )=\frac{\log \sp(T(\gamma _{4},\gamma _{4}))}{\log 3}=\frac{\log 17%
}{\log 3}\text{.}
\end{equation*}%
But $1/2$ is a boundary point of a net interval and has a second symbolic
representation, $\tau =(\gamma _{0},\gamma _{5},\gamma _{12},\gamma
_{12},\gamma _{12}\dots )$, with $\tau \notin L_{e}^{0}$. As $d(\tau )=\log
(17/4)/\log 3$, we have 
\begin{equation*}
\dim _{\mathrm{loc}}\mu (1/2)=\min (d(\tau ),d(\sigma ))=d(\tau ),
\end{equation*}%
so 
\begin{equation*}
\inf_{x\in K_{L}}\{\underline{\dim }_{\mathrm{loc}}\mu (x)\}=\sup_{x\in
K_{L}}\{\overline{\dim }_{\mathrm{loc}}\mu (x)\}=\dim _{\mathrm{loc}}\mu
(1/2)<d(\sigma )=d_{\min }^{L}=d_{\max }^{L}.
\end{equation*}%
We refer the reader to Example \ref{ExSpecial} for more details.
\end{example}

More can be said about the relationship between $d_{\min}^{L}, d_{\max }^{L}$
and local dimensions, but first it is useful to establish that the
convergence to the limiting local behaviour is `uniform' over $\sigma \in
L_{a}$.

\begin{lemma}
\label{dmin}Let $L$ be any set of characteristic vectors containing a loop
class. For each $\varepsilon >0$ there is an integer $k_{0}$ such that if $%
\sigma \in L_{a}$ and $\left\vert \sigma \right\vert \geq k\geq k_{0}$, then 
\begin{equation*}
\frac{\log \left\Vert T(\sigma |k)\right\Vert}{k\log \rho _{\min}}\geq
d_{\min}^{L}-\varepsilon ,
\end{equation*}
equivalently, 
\begin{equation*}
\sup_{\sigma \in L_{a}}\left\Vert T(\sigma |k)\right\Vert ^{1/k}\leq \rho
_{\min}^{d_{\min}^{L}-\varepsilon}.
\end{equation*}
\end{lemma}

\begin{proof}
Fix $s<d_{\min}^{L}$ and let $\mathcal{T}$ be the following set of
transition matrices: 
\begin{equation*}
\mathcal{T}=\{T(\sigma ):\sigma =(\sigma _{1},\dots,\sigma _{n})\in
L_{a},\left\Vert T(\sigma _{1},\dots,\sigma _{j})\right\Vert ^{1/j}>\rho
_{\min}^{s}\text{ for }j=1,\dots,n-1
\end{equation*}
\begin{equation*}
\text{ and }\left\Vert T(\sigma _{1},\dots,\sigma _{n})\right\Vert
^{1/n}\leq \rho _{\min}^{s}\}.
\end{equation*}

If $\mathcal{T}$ is an infinite set, as there are only finitely many
characteristic vectors, there must be infinitely many $T(\sigma )\in 
\mathcal{T}$ with all $\sigma$ having the same first letter, say $\sigma
_{1} $. Among these infinitely many $T(\sigma )$, there must be infinitely
many $\sigma$ all having the same second letter as well, say $\sigma _{2}$.
Repeating this process, we create an infinite path $\sigma =(\sigma
_{1},\sigma _{2},\dots)\in L_{a}$ with $\left\Vert T(\sigma
_{1},\dots,\sigma _{k})\right\Vert ^{1/k}>\rho _{\min}^{s}$ for every $k$.
But then $\underline{d}(\sigma )\leq s<d_{\min}^{L}$ and that is a
contradiction. Consequently, $\mathcal{T}$ is finite.

Let $N$ be the maximal length of any $\alpha$ with $T(\alpha )\in \mathcal{T}
$. Given any finite path $\sigma \in L_{a}$, we can factor $T(\sigma )$ as a
product $\left(\prod\limits_{j=1}^{J}T(\eta _{j})\right)T(\alpha )$ where $%
T(\eta _{j})\in \mathcal{T}$ and $\left\vert \alpha \right\vert \leq N$.
There are only finitely many possible choices for $T(\alpha )$, hence there
is a constant $C$, independent of $\sigma$, such that 
\begin{equation*}
\left\Vert T(\sigma )\right\Vert \leq C\prod_{j=1}^{J}\left\Vert T(\eta
_{j})\right\Vert .
\end{equation*}
As $\left\Vert T(\eta _{j})\right\Vert \leq \rho _{\min}^{s\left\vert \eta
_{j}\right\vert}$ and $\left\vert \alpha \right\vert \leq N$, taking $%
C_{1}=C\rho _{\min}^{-sN}$ gives 
\begin{equation*}
\left\Vert T(\sigma )\right\Vert \leq C\rho _{\min}^{s\sum \left\vert \eta
_{j}\right\vert}=C\rho _{\min}^{s(\left\vert \sigma \right\vert -\left\vert
\alpha \right\vert )}\leq C_{1}\rho _{\min}^{s\left\vert \sigma \right\vert}.
\end{equation*}
Hence for any $\varepsilon >0$, 
\begin{equation*}
\left\Vert T(\sigma |k)\right\Vert \leq C_{1}\rho _{\min}^{sk}\leq \rho
_{\min}^{k(s-\varepsilon )}
\end{equation*}
for sufficiently large $k$. As $s<d_{\min}^{L}$ and $\varepsilon >0$ are
arbitrary choices, this completes the proof.
\end{proof}

Next, we see that a similar result holds for $d_{\max}^{L}$ if we assume $L$
is a loop class.

\begin{lemma}
\label{dmax}Suppose $L$ is a loop class. For each $\varepsilon >0$ there is
an integer $k_{0}$ such that if $\sigma \in L_{a}$ and $\left\vert \sigma
\right\vert \geq k\geq k_{0}$, then 
\begin{equation*}
\frac{\log \left\Vert T(\sigma |k)\right\Vert}{k\log \rho _{\min}}\leq
d_{\max}^{L}+\varepsilon ,
\end{equation*}
equivalently, 
\begin{equation*}
\sup_{\sigma \in L_{a}}\left\Vert T(\sigma |k)\right\Vert ^{1/k}\geq \rho
_{\min}^{(d_{\max}^{L}+\varepsilon )}\text{.}
\end{equation*}
\end{lemma}

\begin{proof}
The proof is quite different. We will use the fact that as $L$ is a loop
class, there is a finite set of paths $\mathcal{S}$ in $L_{a}$ with the
property that given any finite path $\sigma \in L_{a}$, there is some path $%
\beta \in \mathcal{S}$ so that the path $(\sigma ^{-},\beta )$ is a cycle.
Let $C=\max_{\beta \in \mathcal{S}}\left\Vert T(\beta )\right\Vert $.

Given $\sigma \in L_{a}$, put $\sigma _{k}=(\sigma |k)^{-}$. Pick $\beta
_{k}\in \mathcal{S}$ so that $(\sigma _{k},\beta _{k})=\theta _{k}$ is a
cycle. For all positive integers $n$ we have 
\begin{equation*}
\left\Vert (T(\sigma _{k},\beta _{k}))^{n}\right\Vert ^{1/n}\leq \left\Vert
T(\sigma _{k},\beta _{k})\right\Vert \leq \left\Vert T(\sigma |k)\right\Vert
\left\Vert T(\beta _{k})\right\Vert \leq C\left\Vert T(\sigma |k)\right\Vert
\end{equation*}
and hence 
\begin{equation*}
\frac{\log \left\Vert T(\sigma |k)\right\Vert}{k\log \rho _{\min}}\leq \frac{%
\log 1/C+\frac{1}{n}\log \left\Vert (T(\sigma _{k},\beta
_{k}))^{n}\right\Vert}{k\log \rho _{\min}}\text{.}
\end{equation*}
Let $\omega _{k}\in L_{e}^{0}$ be an infinite periodic path with period $%
\theta _{k}$. Then 
\begin{equation*}
d_{\max}^{L}\geq d(\omega _{k})=\frac{\log \sp(T(\theta _{k}))} {%
(k+\left\vert \beta _{k}\right\vert )\log \rho _{\min}}=\lim_{n}\frac{\log
\left\Vert (T(\sigma _{k},\beta _{k}))^{n}\right\Vert}{n(k+\left\vert \beta
_{k}\right\vert )\log \rho _{\min}}.
\end{equation*}
Since $\max_{\beta \in \mathcal{S}}|\beta |<\infty$, given $\varepsilon >0$
there exists $k_{0}=k_{0}(\varepsilon )$ such that for all $k\geq k_{0}$, 
\begin{eqnarray*}
\frac{\log \left\Vert T(\sigma |k)\right\Vert}{k\log \rho _{\min}} &\leq & 
\frac{\log 1/C+\log \sp(T(\theta _{k}))}{k\log \rho _{\min}} \\
&\leq &\frac{\log 1/C}{k\log \rho _{\min}}+d_{\max}^{L}\left( \frac{k
+\left\vert \beta _{k}\right\vert}{k}\right) \leq d_{\max }^{L}+\varepsilon .
\end{eqnarray*}
\end{proof}

We remind the reader that in (\ref{dmin-locdim}) we observed that $%
\underline{\dim}_\mathrm{loc} \mu(x)(x) \leq \underline{d}(\sigma)$ whenever 
$\sigma$ is a symbolic representation of $x$ and thus $\inf_{x\in K_{L}}\{%
\underline{\dim} _{\mathrm{loc}}\mu (x)\}\leq d_{\min}^{L}$. More can be
said, particularly when $K_{L}$ is relatively open.

\begin{proposition}
\label{dminlocdim}Suppose $L$ is a set of vectors containing a loop class.
If $x$ belongs to the relative interior of $K_{L}$, then $\underline{\dim}_%
\mathrm{loc} \mu (x)\geq d_{\min}^{L}$. Moreover, if $K_{L}$ is open (in the
relative topology on $K$), then 
\begin{equation*}
d_{\min}^{L}=\inf \{\underline{\dim}_{\mathrm{loc}}\mu (x):x\in K_{L}\}.
\end{equation*}
\end{proposition}

\begin{remark}
Note that if $K_{L}$ is open, then we necessarily have that $E\subset L$.
\end{remark}

\begin{proof}[Proof of Prop. \protect\ref{dminlocdim}]
First, assume $x\in \mathrm{int}K_{L}$. Choose $N$ so large that $B(x,2\rho
_{\min }^{N})\cap K\subseteq K_{L}$ and let $n\geq N$. Then, for any $n\geq
N $, we have $\Delta _{n}(x)$ and the adjacent level $n$ net intervals, $%
\Delta _{n}^{+}(x)$ and $\Delta _{n}^{-}(x)$, are contained in $B(x,2\rho
_{\min }^{N})$ and thus have symbolic representations of the form $\sigma
|n=(\eta ^{-},\lambda )$ where $\left\vert \eta \right\vert \leq N$ and $%
\lambda \in L_{a}$.

Let $C_{1}=\max \left\Vert T(\eta )\right\Vert$ over the finitely many paths 
$\eta$ of length at most $N$. Lemma \ref{dmin} guarantees that for $n\geq
n(\varepsilon )$, 
\begin{equation*}
\left\Vert T(\sigma |n)\right\Vert \leq \left\Vert T(\eta )\right\Vert
\left\Vert T(\lambda )\right\Vert \leq C_{1}\rho _{\min}^{n(d_{\min
}^{L}-\varepsilon )}
\end{equation*}
for all $\sigma$ of this form. Combined with (\ref{locdimft}), this gives 
\begin{eqnarray*}
\underline{\dim}_{\mathrm{loc}}\mu (x) &=&\liminf_{r\rightarrow 0}\frac{\log
\left( \mu (\Delta _{n}(x))+\mu (\Delta _{n}^{+}(x))+\mu (\Delta
_{n}^{-}(x))\right)}{n\log \rho _{\min}} \\
&\geq &\liminf_{n\rightarrow \infty}\frac{\log \rho _{\min}^{n(d_{\min
}^{L}-\varepsilon )}}{n\log \rho _{\min}}\geq d_{\min}^{L}-\varepsilon
\end{eqnarray*}

Now assume that $K_{L}$ is open. Given any $\varepsilon >0$, choose $\sigma
\in L_{e}^{0}$ with $\underline{d}(\sigma )\leq d_{\min}^{L}+\varepsilon$
and take $x\in K_{L}$ with symbolic representation $\sigma $. Then 
\begin{equation*}
d_{\min}^{L}+\varepsilon \geq \underline{d}(\sigma )\geq \underline{\dim} _{%
\mathrm{loc}}\mu (x)\geq d_{\min}^{L}
\end{equation*}
and that proves $d_{\min}^{L}=\inf_{x\in K_{L}}\left\{ \underline{\dim} _{%
\mathrm{loc}}\mu (x)\right\} $.
\end{proof}

\begin{corollary}
$d_{\min}^{\Omega}=\inf \{\underline{\dim}_{\mathrm{loc}}\mu (x):x\in K\}$.
\end{corollary}

\begin{proof}
This is the special case of $L=\Omega $.
\end{proof}

Let $L$ be a set of characteristic vectors containing a loop class. We will
call a finite 
    path $\eta =(\gamma _{1},\dots,\gamma _{n}) \in L_a$ (with $n>1)$ a
boundary path if all $\gamma _{i}$, for $i>1$, are left-most children of $%
\gamma _{i-1}$, or all are right-most children. We call $\eta$ an \textbf{%
interior path} if it is not a boundary path.
If $L$ is a loop class
and each characteristic vector in $L$ has a unique child in $L$, then all
paths in $L$ will be simple cycles. We will call such a loop class 
\textbf{simple}. If $L$ does not admit an interior path, then it is necessarily
simple (although the converse is not necessarily true).

\begin{proposition}
\label{P:dmax}Suppose $L$ is a loop class.

\begin{enumerate}
\item If $L$ admits an interior path, then 
\begin{equation}
d_{\max}^{L}=\sup \{\overline{\dim}_{\mathrm{loc}}\mu (x):x\in K_{L}\}=\sup
\{\dim _{\mathrm{loc}}\mu (x):x\in K_{L}\}.  \label{dmax=}
\end{equation}
\label{P:dmax i}

\item Otherwise, $L$ is simple and in this case 
\begin{equation*}
d_{\max}^{L}=\frac{\log \sp(T(\theta (L)))}{\left\vert L\right\vert \log
\rho _{\min}}
\end{equation*}
(with the notation $\theta (L)$ introduced above). \label{P:dmax ii}
\end{enumerate}
\end{proposition}

\begin{proof}
Part \ref{P:dmax i}. Any loop class which admits an interior path has
interior paths (in the loop class), which join any two members of the class.
For each pair $\alpha ,\beta \in L$, pick one such interior path in $L$ and
call this finite set of interior paths $\mathcal{P}$.

Fix $\varepsilon >0$ and choose $\sigma \in L_{a}$ such that $\overline{d}%
(\sigma )\geq d_{\max }^{L}-\varepsilon /2$. Then select a subsequence $%
(n_{k})$ such that $\left\Vert T(\sigma |n_{k})\right\Vert \leq \rho _{\min
}^{n_{k}(d-\varepsilon )}$ where $d=d_{\max}^{L}$. Let $\sigma _{k}=\sigma
|n_{k}$ and choose a path $\lambda _{k}\in \mathcal{P}$ such that $\theta
_{k}=$ $(\sigma _{k}^{-},\lambda _{k})$ is a cycle. Let $x$ be a periodic
point in $K_{L}$ with symbolic representation having period $\theta _{k}$.

As $\lambda _{k}$ is an interior path, this symbolic representation of $x$
is unique. As per Example \ref{Ex:periodic}, the local dimension at $x$
exists and is given by 
\begin{equation*}
\dim _{\mathrm{loc}}\mu (x)=\frac{\log \sp(T(\theta _{k}))}{\left\vert
\theta _{k}^{-}\right\vert \log \rho _{\min}}.
\end{equation*}
For any $n$, the submultiplicativity of the norm implies 
\begin{equation*}
\left\Vert \left( T(\theta _{k})\right) ^{n}\right\Vert ^{1/n}\leq
\left\Vert T(\sigma |n_{k})\right\Vert \left\Vert T(\lambda _{k})\right\Vert
\leq C\rho _{\min}^{n_{k}(d-\varepsilon )}
\end{equation*}
where $C=\max_{\lambda \in \mathcal{P}}\left\Vert T(\lambda )\right\Vert $.
Thus if $n_{k}$ is sufficiently large, then 
\begin{equation*}
\frac{\log \sp(T(\theta _{k}))}{\left\vert \theta _{k}^{-}\right\vert \log
\rho _{\min}}=\lim_{n\rightarrow \infty}\frac{\log \left\Vert (T(\theta
_{k}))^{n}\right\Vert ^{1/n}}{(n_{k}+\left\vert \lambda _{k}\right\vert
)\log \rho _{\min}}\geq \frac{\log C\rho _{\min}^{n_{k}(d-\varepsilon )}} {%
n_{k}\log \rho _{\min}}\geq d-2\varepsilon .
\end{equation*}
As $\varepsilon >0$ was arbitrary, it follows that $\sup_{x\in K_{L}}\{\dim
_{\mathrm{loc}}\mu (x)\}\geq d$. Since we previously saw that $\sup_{x\in
K_{L}} \{\overline{\dim}_{\mathrm{loc}}\mu (x)\}\leq d_{\max}^{L}$ when $L$
is a loop class, this is enough to establish both equalities in (\ref{dmax=}%
).

Part \ref{P:dmax ii}. We have already noted that if $L$ does not admit any
interior paths, then the transition graph is the simple cycle $\theta (L)$.
Thus any $\sigma \in L_{e}^{0}$ is periodic with period $\theta (L)$ and $%
d(\sigma )=\log spT(\theta (L))/\left\vert L\right\vert \log \rho
_{\min}=d_{\max}^{L}$.
\end{proof}

Here are some immediate corollaries.

\begin{corollary}
Let $E$ be the essential class. Then $d_{\max}^{E}=\sup \{\dim _{\mathrm{loc}%
}\mu (x):x\in K_{E}\}$. \label{Cor:dE}
\end{corollary}

\begin{remark}
We note that if $L$ is a loop class and $K_L$ is relatively open in $K$ then 
$L = E$.
\end{remark}

\begin{proof}[Proof of Cor. \protect\ref{Cor:dE}]
The essential class admits an interior path since every child of a
characteristic vector in $E$ is again in $E$.
\end{proof}

Another consequence of Proposition \ref{P:dmax} is that 
\begin{equation*}
d_{\max}^{\Omega}=\max \left\{ \sup_{x\in K}\{\dim _{\mathrm{loc}}\mu (x)\},%
\text{ } \max_{L\text{ simple, max loop class}}\left\{ \frac{\log \sp%
(T(\theta (L)))} {\left\vert L\right\vert \log \rho _{\min}}\right\}
\right\} .
\end{equation*}
However, $d_{\max}^{\Omega}$ $>\sup_{x\in K}\{\dim _{\mathrm{loc}}\mu (x)\}$
is possible, as the next example illustrates.

\begin{example}
Consider, again, the IFS of Example \ref{ExSp1}. We observed there that $%
d_{\max}^{L}=\log 17/\log 3$ for $L=\{\gamma _{4}\}$. For each of the other
singleton maximal loop classes $L$, one can compute that $d_{\max}^{L}=\log
(17/4)/\log 3$, so $d_{\max}^{\Omega}=\log 17/\log 3$. With more work it can
be shown that $\overline{\dim}_{\mathrm{loc}}\mu (x)\leq \log (17/2)/\log 3$
for all $x\in \lbrack 0,1]$. See Example \ref{ExSpecial} for more details.
\end{example}

\section{Loop class spectrum\label{sp}}

The $L^{q}$-spectrum can be defined for any measure.

\begin{definition}
\label{D:tau} The $L^{q}$-\textbf{spectrum} of the measure $\mu$ is defined
to be the function $\tau (\mu ,q)$ defined on $\mathbb{R}$ by 
\begin{equation*}
\tau (\mu ,q)=\liminf_{\delta \rightarrow 0}\frac{\log \sup \sum_{i}\left(
\mu (B(x_{i},\delta ))\right) ^{q}}{\log \delta},
\end{equation*}
where the supremum is over all countable collections of disjoint open balls, 
$B(x_{i},\delta )$, with centres $x_{i}\in \mathrm{supp} \mu$. We write $%
\tau (q)$ if the measure $\mu$ is clear.
\end{definition}

\begin{remark}
This is also known in the literature as the lower $L^{q}$-spectrum. The
upper $L^{q}$-spectrum (and $L^{q}$-spectrum) could be similarly defined,
but as these will not be of interest to us, we will refer to $\tau$ as the $%
L^{q}$-spectrum, for short.
\end{remark}

As proven in \cite[Section 3]{LN}, for any measure $\mu$, we have $\tau (\mu
,\cdot )$ is an increasing concave function. Further it does not take on
either of the values $\pm \infty$ if and only if 
\begin{equation}
\liminf_{\delta \rightarrow 0}\frac{\log (\inf_{x\in \mathrm{supp} \mu}\mu
(B(x,\delta )))}{\log \delta}<\infty \text{.}  \label{LN}
\end{equation}

We show in Prop. \ref{tauformula} that this is indeed the case for measures
of finite type. For finite type measures $\mu $, balls centred in the
support of $\mu $ can be replaced by net intervals in the definition of the $%
L^{q}$-spectrum. This was shown by Feng for equicontractive finite type
measures, \cite[Prop. 5.6]{FengSmoothness}. The proof for the general case
is similar and is included here for completeness.

\begin{proposition}
\label{tauformula}Suppose $\mu $ is a measure of finite type. The $L^{q}$%
-spectrum of $\mu $ can be computed as 
\begin{equation}
\tau (\mu ,q)=\liminf_{n\rightarrow \infty }\frac{\log \sum_{\Delta \in 
\mathcal{F}_{n}}(\mu (\Delta ))^{q}}{n\log \rho _{\min }}=\liminf_{n%
\rightarrow \infty }\frac{\log \sum_{\left\vert \sigma \right\vert =n,\sigma
\in \Omega _{a}^{0}}\left\Vert T(\sigma )\right\Vert ^{q}}{n\log \rho _{\min
}}.  \label{LspReln}
\end{equation}%
Moreover, $\tau (\mu ,q)$ is real-valued for all $q$.
\end{proposition}

\begin{proof}
We first remark that as $\mu (\Delta )$ is comparable to $\left\Vert
T(\sigma )\right\Vert$ when $\Delta$ has symbolic representation $\sigma$,
the second equality of the display is clear.

Given $0<\delta <\rho _{\min}$, choose the integer $n$ such that $\rho
_{\min}^{n}\leq \delta <\rho _{\min}^{n-1}$. Then, for any $x\in \mathrm{supp%
} \mu$, $B(x,\delta )$ contains $\Delta _{n}(x)$. Applying (\ref{munet}), it
follows that there are positive constants $C,s$ such that $\mu (B(x,\delta
))\geq \mu (\Delta _{n})\geq C\left( \min p_{j}\right) ^{sn}$. That
certainly implies the left side of (\ref{LN}) is finite and hence $\tau (q)$
is real valued.

Furthermore, if $q<0$, then 
\begin{equation*}
\left( \mu (B(x,\delta ))\right) ^{q}\leq \mu (\Delta _{n})^{q},
\end{equation*}%
while if $q\geq 0$, then 
\begin{equation*}
\left( \mu (B(x,\delta ))\right) ^{q}\leq \left( \sum_{\substack{ \text{%
level }n\text{ net intervals }\Delta _{n}  \\ \text{intersecting }B(x,\delta
) }}\mu (\Delta _{n})\right) ^{q}.
\end{equation*}%
Moreover, the fact that $\delta <\rho _{\min }^{n-1}$ ensures that $%
B(x,\delta )$ is contained in the union of the at most $2/(c\rho _{\min
})=C_{0}$ net intervals of level $n$ that intersect it. Since each net
interval of level $n$ can intersect at most two balls of radius $\delta $,
we have 
\begin{equation}
\sum_{i}\left( \mu (B(x_{i},\delta ))\right) ^{q}\leq \left\{ 
\begin{array}{cc}
2C_{0}^{q}\sum_{\Delta \in \mathcal{F}_{n}}\mu (\Delta )^{q} & \text{if }%
q\geq 0 \\ 
\sum_{\Delta \in \mathcal{F}_{n}}\mu (\Delta )^{q} & \text{if }q<0%
\end{array}%
\right. .  \label{sp2}
\end{equation}%
It is easy to see from (\ref{sp2}) that for either choice of $q$, 
\begin{equation*}
\tau (q)\geq \liminf_{n\rightarrow \infty }\frac{\log \sum_{\Delta \in 
\mathcal{F}_{n}}(\mu (\Delta ))^{q}}{n\log \rho _{\min }}.
\end{equation*}

For the other inequality, we argue as follows. Choose $N_{0}$ such that $%
\rho _{\min}^{N_{0}}\leq c/3$. Consider any level $n$ net interval $\Delta
=[a,b]$ and its neighbour set $((a_{j},L_{j}))_{j=1}^{J}$. For each $j$,
choose $\sigma _{j}\in \Lambda _{n}$ associated with $(a_{j},L_{j})$ and $%
\Delta $. Select the index $i\in \{1,\dots,J\}$ so that $\sum_{\alpha
:S_{\alpha}=S_{\sigma _{j}}}p_{\alpha}$ is maximal for $j=i$.

Obtain an integer $N\geq N_{0}$ and finite set $\mathcal{W}$, as in Lemma %
\ref{incl} \ref{incl i}, with the property that there is some $\nu \in 
\mathcal{W}$ with $S_{\sigma _{i}\nu }[0,1]\subseteq \Delta $ and $\sigma
_{i}\nu \in \Lambda _{n+N}$. As the length of this interval satisfies 
\begin{equation*}
\rho _{\min }^{n+N+1}\leq \ell (S_{\sigma _{i}\nu }[0,1])\leq \ell (\Delta
)/3,
\end{equation*}%
one of the two endpoints of the interval $S_{\sigma _{i}\nu }[0,1]$, say $%
x_{\Delta }$, has distance at least $\ell (S_{\sigma _{i}\nu }[0,1])$ to
both $a$ and $b$. Of course, $x_{\Delta }$ belongs to $K$. Therefore 
\begin{equation*}
S_{\sigma _{i}\nu }[0,1]\subseteq B(x_{\Delta },\ell (S_{\sigma _{i}\nu
}[0,1]))\subseteq \Delta .
\end{equation*}

If we let $M$ be the maximum number of elements in any neighbour set, then
the maximality property in the choice of index $i$ means that 
\begin{equation*}
\mu (\Delta )\geq \mu (B(x_{\Delta},\ell (S_{\sigma _{i}\nu}[0,1])))\geq
\sum_{\omega :S_{\omega}=S_{\sigma _{i}v}}p_{\omega}\geq p_{\nu
}\sum_{\alpha :S_{\alpha}=S_{\sigma _{i}}}p_{\alpha}\geq p_{\nu}\frac{\mu
(\Delta )}{M}.
\end{equation*}
As there are only finitely many such words $v$, there is some $\varepsilon
>0 $ such that $p_{v}>\varepsilon$ for all such $v$. Hence $\mu
(B(x_{\Delta},\ell (S_{\sigma _{i}\nu}[0,1])))\geq \mu (\Delta )\varepsilon
/M$ for all $\Delta $.

Since the balls $B(x_{\Delta},\ell (S_{\sigma _{i}\nu}[0,1])_{\min }^{n+N})$
are disjoint, centred in $K$ and have radius comparable to $\rho _{\min}^{n}$%
, we can now conclude that for all $q$, 
\begin{equation*}
\tau (q)\leq \liminf_{n\rightarrow \infty}\frac{\log \sum_{\Delta \in 
\mathcal{F}_{n}}(\mu (\Delta ))^{q}}{n\log \rho _{\min}}.
\end{equation*}
\end{proof}

Motivated by this, we make the following definition for the loop class $%
L^{q} $-spectrum of a measure $\mu$ of finite type.

\begin{definition}
Let $\mu$ be a measure of finite type and let $L$ be a set of characteristic
vectors containing a loop class. We define the $L^{q}$-\textbf{spectrum of} $%
\mu$ \textbf{on} $L$ as the function $\tau _{L}(\mu ,q)$ defined at $q\in 
\mathbb{R}$ by 
\begin{equation*}
\tau _{L}(\mu ,q)=\liminf_{k\rightarrow \infty}\frac{\log \sum_{\sigma \in
L_{a},\left\vert \sigma \right\vert =k}\left\Vert T(\sigma )\right\Vert ^{q}%
} {k\log \rho _{\min}}\text{ .}
\end{equation*}
Again, we suppress $\mu$ in the notation if the measure is clear.
\end{definition}

We will first prove that $\tau _{\Omega}(\mu ,q)=\tau (\mu ,q)$ and $\tau
_{\Omega}$ is minimal over all $\tau _{L}$.

\begin{theorem}
\label{tauletauL}Let $L\subseteq \Omega$ be any set of characteristic
vectors containing a loop class. Then $\tau (\mu ,q)=\tau _{\Omega }(\mu
,q)\leq \tau _{L}(\mu ,q)$ for all $q$ and hence 
\begin{equation*}
\tau (\mu ,q)\leq \min \{\tau _{L}(\mu ,q):L\text{ maximal loop class}\} 
\text{ for all }q\in \mathbb{R}.
\end{equation*}
\end{theorem}

\begin{proof}
We first check that $\tau (\mu ,q)=\tau _{\Omega }(\mu ,q)$. {}Since 
\begin{equation*}
\sum_{\sigma \in \Omega _{a}^{0},\left\vert \sigma \right\vert =k}\left\Vert
T(\sigma )\right\Vert ^{q}\leq \sum_{\sigma \in \Omega _{a},\left\vert
\sigma \right\vert =k}\left\Vert T(\sigma )\right\Vert ^{q},
\end{equation*}
it is immediate that $\tau (\mu ,q)\geq \tau _{\Omega }(\mu ,q)$.

To prove the other inequality, we start by partitioning the right hand sum
according to the first letter of $\sigma $. Let $\Omega =\{\gamma
_{0},\gamma _{1},\dots,\gamma _{\left\vert \Omega \right\vert -1}\}$ be the
complete list of characteristic vectors and for each $i$, choose $\eta _{i}$
such that $(\gamma _{0},\eta _{i},\gamma _{i})$ is an admissible path where $%
\left\vert \eta _{i}\right\vert =n_{i}$ (and $(\eta _{0},\gamma _{0})$
should be understood to be the empty word).

As $\left\Vert T(\gamma _{i},\chi )\right\Vert$ is comparable to $\left\Vert
T(\gamma _{0},\eta _{i},\gamma _{i},\chi )\right\Vert $, there is a constant 
$C_{i}(q)$ such that 
\begin{eqnarray*}
\sum_{\sigma \in \Omega _{a},\left\vert \sigma \right\vert =k}\left\Vert
T(\sigma )\right\Vert ^{q} &=&\sum_{i=0}^{\left\vert \Omega \right\vert
-1}\sum_{\left\vert \chi \right\vert =k-1}\left\Vert T(\gamma _{i},\chi
)\right\Vert ^{q} \\
& \leq & \sum_{i=0}^{\left\vert \Omega \right\vert -1}C_{i}\sum_{\left\vert
\chi \right\vert =k-1}\left\Vert T(\gamma _{0},\eta _{i},\gamma _{i},\chi
)\right\Vert ^{q} \\
&\leq &\sum_{i=0}^{\left\vert \Omega \right\vert -1}C_{i}\sum_{\left\vert
\sigma \right\vert =k+n_{i}+1,\sigma \in \Omega _{a}^{0}}\left\Vert T(\sigma
)\right\Vert ^{q}.
\end{eqnarray*}
The definition of $\tau (\mu ,q)$ implies that given any $\varepsilon >0$
there is some $k_{0}=k_{0}(\varepsilon ,q)$ such that for all $k\geq k_{0}$
we have 
\begin{equation*}
\frac{\log \sum_{\left\vert \sigma \right\vert =k,\sigma \in \Omega
_{a}^{0}}\left\Vert T(\sigma )\right\Vert ^{q}}{k\log \rho _{\min}}\geq \tau
(\mu ,q)-\varepsilon .
\end{equation*}
Thus for $k$ sufficiently large 
\begin{equation*}
\sum_{\left\vert \sigma \right\vert =k+n_{i}+1,\sigma \in \Omega
_{a}^{0}}\left\Vert T(\sigma )\right\Vert ^{q}\leq \rho _{\min
}^{(k+n_{i}+1)(\tau (q)-\varepsilon )}.
\end{equation*}
We deduce that 
\begin{equation*}
\sum_{\sigma \in \Omega _{a},\left\vert \sigma \right\vert =k}\left\Vert
T(\sigma )\right\Vert ^{q}\leq \sum_{i=0}^{\left\vert \Omega \right\vert
-1}C_{i}^{\prime }\rho _{\min}^{k(\tau (q)-\varepsilon )}\leq C\rho _{\min
}^{k(\tau (q)-\varepsilon )}
\end{equation*}
for suitable constants $C_{i}^{\prime },C$ depending on $\varepsilon$ and $q$%
. It is immediate from this that 
\begin{equation*}
\frac{\log \sum_{\sigma \in \Omega _{a},\left\vert \sigma \right\vert
=k}\left\Vert T(\sigma )\right\Vert ^{q}}{k\log \rho _{\min}}\geq \frac{\log
C}{k\log \rho _{\min}}+\tau (\mu ,q)-\varepsilon
\end{equation*}
and letting $k\rightarrow \infty $, we see that $\tau _{\Omega }(\mu ,q)\geq
\tau (\mu ,q)-\varepsilon$ for all $\varepsilon >0$. Thus $\tau _{\Omega
}(\mu ,q)=\tau (\mu ,q)$ for all $q$.

For any $L\subseteq \Omega$, 
\begin{equation*}
\sum_{\sigma \in L_{a},\left\vert \sigma \right\vert =k}\left\Vert T(\sigma
)\right\Vert ^{q}\leq \sum_{\sigma \in \Omega _{a},\left\vert \sigma
\right\vert =k}\left\Vert T(\sigma )\right\Vert ^{q},
\end{equation*}
hence $\tau _{L}(\mu ,q)\geq \tau _{\Omega }(\mu ,q)$ for all $q$,
completing the proof.
\end{proof}

We next obtain pointwise bounds for the functions $\tau _{L}$. By the
incidence matrix of a set of characteristic vectors $L=\{\chi
_{1},\dots,\chi _{\left\vert L\right\vert }\}$ we mean the $\left\vert
L\right\vert \times \left\vert L\right\vert$ matrix whose $(i,j)$ entry is $%
1 $ if $\chi _{i}$ has $\chi _{j}$ as a child and equals $0$ otherwise.

\begin{proposition}
\label{taubound}Let $L$ be a set of characteristic vectors containing a loop
class and let $I_{L}$ denote its incidence matrix. Then 
\begin{equation*}
qd_{\min}^{L}\geq \tau _{L}(\mu ,q)\geq qd_{\min}^{L}-\frac{\log \sp(I_{L})%
} {\left\vert \log \rho _{\min}\right\vert }\text{ if }q\geq 0.
\end{equation*}
If $L$ is a loop class, then 
\begin{equation*}
qd_{\max}^{L}\geq \tau _{L}(\mu ,q)\geq qd_{\max}^{L}-\frac{\log \sp(I_{L})%
} {\left\vert \log \rho _{\min}\right\vert }\text{ if }q<0\text{.}
\end{equation*}
\end{proposition}

\begin{proof}
Let $\varepsilon >0$ and assume $L=\{\chi _{1},\dots,\chi _{\left\vert
L\right\vert }\}$ is a loop class. By Lemma \ref{dmax}, for sufficiently
large $k$ and $\sigma \in L_{a}, \left\vert \sigma \right\vert =k$, we have 
\begin{equation*}
\left\Vert T(\sigma )\right\Vert \geq \rho _{\min}^{(d_{\max
}^{L}+\varepsilon )k}.
\end{equation*}
Furthemore, we can choose an infinite path $\alpha \in L_{a}$ such that $%
\overline{d}(\alpha )\geq d_{\max}^{L}-\varepsilon /2$, hence for infinitely
many $k_{j}$ we have 
\begin{equation}
\left\Vert T(\alpha |k_{j})\right\Vert \leq \rho _{\min}^{k_{j}(d_{\max
}^{L}-\varepsilon )}.  \label{al}
\end{equation}

Assume $\chi _{i}$ is the characteristic vector of a net interval $\Delta
_{i}$ of level $n_{i}$. Every $\sigma \in L_{a}$ with $\left\vert \sigma
\right\vert =k+1$ and beginning with letter $\chi _{i}$, determines a unique
descendent of $\Delta _{i}$ at level $n_{i}+k$ that has its characteristic
vector in $L$. The number of such $\sigma$ is the sum of the entries on row $%
i$ of the matrix $(I_{L})^{k}$. This is dominated by $\left\Vert
(I_{L})^{k}\right\Vert$ and hence by $(\sp(I_{L})+\varepsilon )^{k}$ for
large enough $k$ depending on $\varepsilon >0$.

Thus for $q<0$ and large $k$, 
\begin{equation*}
\sum_{\sigma \in L_{a},\left\vert \sigma \right\vert =k}\left\Vert T(\sigma
)\right\Vert ^{q}=\sum_{i=1}^{\left\vert L\right\vert }\sum_{\left\vert
\sigma \right\vert =k-1}\left\Vert T(\chi _{i},\sigma )\right\Vert ^{q}\leq
\rho _{\min}^{qk(d_{\max}^{L}+\varepsilon )}(\sp(I_{L})+\varepsilon )^{k},
\end{equation*}
while for $\alpha$ as in (\ref{al}) and infinitely many $k$, we have 
\begin{equation*}
\sum_{\sigma \in L_{a},\left\vert \sigma \right\vert =k}\left\Vert T(\sigma
)\right\Vert ^{q}\geq \left\Vert T(\alpha |k)\right\Vert ^{q}\geq \rho
_{\min}^{kq(d_{\max}^{L}-\varepsilon )}.
\end{equation*}
Taking logarithms, dividing by $k\log \rho _{\min}$ and letting $%
k\rightarrow \infty$ gives 
\begin{equation*}
\tau _{L}(q)=\liminf_{k\rightarrow \infty }\frac{\log \sum_{\sigma \in
L_{a},\left\vert \sigma \right\vert =k}\left\Vert T(\sigma )\right\Vert ^{q}%
} {k\log \rho _{\min}}\left\{ 
\begin{array}{cc}
\geq & q(d_{\max}^{L}+\varepsilon )-\frac{\log (\sp(I_{L})+\varepsilon )} {%
\left\vert \log \rho _{\min}\right\vert } \\ 
\leq & q(d_{\max}^{L}-\varepsilon )%
\end{array}
\right. ,
\end{equation*}
which proves the claim for $q<0$.

The argument for $q\geq 0$ is similar using Lemma \ref{dmin}, instead of
Lemma \ref{dmax}, and only requires $L$ to contain a loop class.
\end{proof}

It was proven by Ngai and Wang in \cite{NW} that $\dim _{B}K=\log (\sp%
(I_{\Omega }))/\left\vert \log \rho _{\min}\right\vert $. More generally,
the following is true.

\begin{lemma}
Suppose $L$ is a set of characteristic vectors containing a loop class and
let $J_{L}$ be the incidence matrix for the set of characteristic vectors
that have a descendent in $L$. The box dimension of $K_{L}$ exists and is
equal to 
\begin{equation*}
\dim _{B}K_{L}=\frac{\log \sp(J_{L})}{\left\vert \log \rho _{\min
}\right\vert }.
\end{equation*}
\end{lemma}

\begin{remark}
Note that $J_{L}$ above is not the incidence matrix of $L$. It may contain
characterstic vectors outside of $L$ that have descendents in $L$. It is
possible that $\sp(J_{L})>\sp(I_{L})$ as we show in Example \ref{ex:J vs I}.
\end{remark}

\begin{proof}
Given a matrix $J$, let $\left\Vert J\right\Vert _{i}=\sum_{j}\left\vert
J_{ij}\right\vert$, the sum of the modulos of the entries of row $i$ of $J$.

Without loss of generality, we can assume the entries of the first row of $%
J_{L}$ are determined by the children of $\gamma _{0}$ (with descendents in $%
L$). Then $\left\Vert \left( J_{L}\right) ^{k}\right\Vert _{1}$ is the
number of net intervals at level $k$ that contain an element of $K_{L}$.
This is an increasing function of $k$ and is comparable to the number of
disjoint balls of radius $c\rho _{\min}^{k}$ centred at points in $K_{L}$.

Furthermore, $\left\Vert J_{L}^{k}\right\Vert _{i}\leq \left\Vert
J_{L}^{k+m_{i}}\right\Vert _{1}$ if the entries of row $i$ of $J$ are
determined by the children of $\gamma _{i}$ where $\gamma _{i}$ is a
descendent of $\gamma _{0}$ at level $m_{i}$ which has a descendent in $L$.
Thus if $C$ is the number of rows of $J_{L}$ and $M=\max m_{i}$, we have 
\begin{equation*}
\left\Vert J_{L}^{k}\right\Vert _{1}\leq \left\Vert J_{L}^{k}\right\Vert
=\sum_{i}\left\Vert J_{L}^{k}\right\Vert _{i}\leq C\left\Vert
J_{L}^{k+M}\right\Vert _{1}.
\end{equation*}%
This shows $\lim_{k\rightarrow \infty }\left\Vert J_{L}^{k}\right\Vert
_{1}^{1/k}=\lim_{k}\left\Vert J_{L}^{k}\right\Vert ^{1/k}=\sp(J_{L})$ and
therefore 
\begin{equation*}
\frac{\log \sp(J_{L})}{\left\vert \log \rho _{\min }\right\vert }%
=\lim_{k\rightarrow \infty }\frac{\log \left\Vert J_{L}^{k}\right\Vert _{1}}{%
k\left\vert \log \rho _{\min }\right\vert }=\dim _{B}K_{L}.
\end{equation*}
\end{proof}

\begin{corollary}
\label{box} Let $L$ be a set of characteristic vectors containing a loop
class.

\begin{enumerate}
\item Then 
\begin{equation*}
qd_{\min}^{L}\geq \tau _{L}(\mu ,q)\geq qd_{\min}^{L}-\dim _{B}K_{L} \geq q
d_{\min}^L - 1 \text{ if }q\geq 0.
\end{equation*}
\label{box i}

\item Further, if $L$ is a loop class, then 
\begin{equation*}
qd_{\max }^{L}\geq \tau _{L}(\mu ,q)\geq qd_{\max }^{L}-\dim _{B}K_{L}\geq
qd_{\max }^{L}-1\text{ if }q<0\text{.}
\end{equation*}%
\label{box ii}

\item \label{box iv} For any $q\geq 0$, 
\begin{equation*}
qd_{\min }^{\Omega }\geq \tau (\mu ,q)\geq qd_{\min }^{\Omega }-\dim _{B}K%
\text{ }\geq qd_{\min }^{\Omega }-1.
\end{equation*}

\item If $|L|=1$ then \label{box iii} 
\begin{equation*}
qd_{\min }^{L}=qd_{\max }^{L}=\tau _{L}(\mu ,q)\text{ for all }q.
\end{equation*}
\end{enumerate}
\end{corollary}

\begin{proof}
Part \ref{box i} and \ref{box ii}. Let $J_{L}$ be as in the Lemma. Then
the incidence matrix of $L$, $I_{L}$, is a submatrix of $J_{L}$ and $%
\left\Vert I_{L}^{k}\right\Vert \leq \left\Vert J_{L}^{k}\right\Vert$ for
all $k$, so $\sp(I_{L})\leq \sp(J_{L})$.

Part \ref{box iv} is the special case of $L=\Omega $ in Part (i).

Part \ref{box iii}. Assume $L=\{\gamma \}$. We have 
\begin{equation*}
d_{\min }^{L}=d_{\max }^{L}=\frac{\log \sp(T(\gamma ,\gamma ))}{\log \rho
_{\min }}
\end{equation*}%
and $I_{L}=[1]$, so $sp(I_{L})=1$.
\end{proof}

%\begin{corollary}
%For any $q\geq 0$,
%\begin{equation*}
%qd_{\min}^{\Omega }\geq \tau (\mu ,q)\geq qd_{\min}^{\Omega }-\dim_{B}K
%\text{ }\geq qd_{\min}^{\Omega }-1.
%\end{equation*}
%\end{corollary}

We remark that it is possible to have $\sp(I_{L})<\sp(J_{L})$. Here is an
example.

\begin{example}
\label{ex:J vs I} Consider the finite type IFS with maps $S_{j}(x)=x/3+d_{j}$
where $d_{j}=0,4/9,5/9,2/3$, discussed in Example 3.10 of \cite{HHN}. There
are seven characteristic vectors: $\gamma _{0}$ with children $\gamma
_{i},i=0,1,\dots,5;$ $\gamma _{1}$ with child $\gamma _{0};$ $\gamma _{2}$
with children $\gamma _{1},\gamma _{2},\gamma _{3};$ $\gamma _{5}$ with
children $\gamma _{3},\gamma _{4},\gamma _{5};$ and $\gamma _{3},\gamma
_{4},\gamma _{6},\gamma _{7}$ all with children $\gamma _{3},\gamma
_{6},\gamma _{7}$.

The singleton $L=\{\gamma _{5}\}$ is a maximal loop class. Here $I_{L}=[1]$
with $\sp(I_{L})=1$, while 
\begin{equation*}
J_{L}= 
\begin{bmatrix}
1 & 1 & 1 & 1 \\ 
1 & 0 & 0 & 0 \\ 
0 & 1 & 1 & 0 \\ 
0 & 0 & 0 & 1%
\end{bmatrix}%
\end{equation*}
with $\sp(J_{L})=2$.
\end{example}

In \cite[Thm. 1.2]{FL} it is shown that for any self-similar measure
satisfying the weak separation property, $\tau (q)/q\rightarrow \inf \{%
\underline{\dim}_{\mathrm{loc}}\mu (x):x\}$ as $q\rightarrow \infty $. More
generally, we also can also immediately deduce the following by combining
Propositions \ref{dminlocdim}, \ref{P:dmax} and \ref{taubound}.

\begin{corollary}
\label{asy} Let $L$ be a set of characteristic vectors containing a loop
class.

\begin{enumerate}
\item We have $\tau _{L}(q)/q\rightarrow d_{\min}^{L}$ as $q\rightarrow
\infty $. If $K_{L}$ is relatively open, then 
\begin{equation*}
\tau _{L}(q)/q\rightarrow \inf \{\underline{\dim}_{\mathrm{loc}}\mu (x):x\in
K_{L}\} \text{as }q\rightarrow \infty .
\end{equation*}
\label{asy i}

\item Further, if $L$ is a loop class, then $\tau _{L}(q)/q\rightarrow
d_{\max }^{L}$ as $q\rightarrow -\infty $. If, in addition, $L$ admits an
interior path, then 
\begin{equation*}
\tau _{L}(q)/q\rightarrow \sup \{\dim _{\mathrm{loc}}\mu (x):x\in K_{L}\}%
\text{ as }q\rightarrow -\infty .
\end{equation*}%
\label{asy ii}
\end{enumerate}
\end{corollary}

Here are some other facts about loop class $L^{q}$-spectra which will be
useful later in the paper.

\begin{proposition}
Let $E$ be the essential class and suppose $\Delta ^{(j)},j=1,\dots ,J$, are
distinct net intervals of level $N$ with characteristic vectors in $E$. Let $%
\mu _{j}=\mu |_{\Delta ^{(j)}}$. For any real number $q$, 
\begin{equation*}
\tau _{E}(\mu ,q)=\tau (\mu _{1}+\cdot \cdot \cdot +\mu _{J},q).
\end{equation*}%
Moreover, 
\begin{equation*}
\tau (\mu ,q)=\tau _{E}(\mu ,q)\text{ if }q\geq 0.
\end{equation*}
\end{proposition}

\begin{proof}
We first note that since the net intervals $\Delta ^{(j)}$ are disjoint,
similar arguments to the proof of Proposition \ref{tauformula} show that 
\begin{equation*}
\tau (\mu _{1}+\cdot \cdot \cdot +\mu _{J},q)=\liminf_{k\rightarrow \infty } 
\frac{\log \sum_{j=1}^{J}\sum_{\Delta \subseteq \Delta ^{(j)},\Delta \in 
\mathcal{F}_{k}}(\mu _{j}(\Delta ))^{q}}{k\log \rho _{\min}}.
\end{equation*}

Suppose $\Delta ^{(j)}$ has symbolic representation $\delta ^{(j)}$, with
last letter $\gamma _{j}\in E$. The usual comparability arguments give 
\begin{eqnarray}
\tau (\mu _{1}+\cdot \cdot \cdot +\mu _{J},q) &=&\liminf_{k\rightarrow
\infty }\frac{\log \sum_{j=1}^{J}\sum_{|\sigma |=k}\left\Vert T(\delta
^{(j)},\sigma )\right\Vert }{k\log \rho _{\min}}^{q}  \label{Delta1} \\
&=&\liminf_{k\rightarrow \infty }\frac{\log \sum_{j=1}^{J}\sum_{|\sigma
|=k}\left\Vert T(\gamma _{j},\sigma )\right\Vert ^{q}}{k\log \rho _{\min}} 
\notag
\end{eqnarray}
As $\gamma _{j}\in E$, the admissible paths $\sigma$ appearing in the sum
belong to $E_{a}$ and therefore 
\begin{equation*}
\sum_{|\sigma |=k}\left\Vert T(\gamma _{j},\sigma )\right\Vert ^{q}\leq
\sum_{|\sigma |=k+1,\sigma \in E_{a}}\left\Vert T(\sigma )\right\Vert ^{q} 
\text{.}
\end{equation*}
It follows that 
\begin{equation*}
\tau (\mu _{1}+\cdot \cdot \cdot +\mu _{J},q)\geq \liminf_{k\rightarrow
\infty }\frac{\log \sum_{|\sigma |=k,\sigma \in E_{a}}\left\Vert T(\sigma
)\right\Vert ^{q}}{k\log \rho _{\min}}=\tau _{E}(\mu ,q).
\end{equation*}

The proof that $\tau _{E}(\mu ,q)\geq$ $\tau (\mu _{1}+\cdot \cdot \cdot
+\mu _{J},q)$ uses arguments similar to those used in the proof of Theorem %
\ref{tauletauL}. Assume $E=\{\chi _{1},\dots,\chi _{\left\vert E\right\vert
}\} $. As every $\sigma \in E_{a}$ will begin with one of the $\chi _{i}$,
we can write 
\begin{equation*}
\sum_{|\sigma |=k,\sigma \in E_{a}}\left\Vert T(\sigma )\right\Vert
^{q}=\sum_{i=1}^{\left\vert E\right\vert }\sum_{|\sigma |=k-1,\sigma \in
E_{a}}\left\Vert T(\chi _{i},\sigma )\right\Vert ^{q}.
\end{equation*}

We continue to assume $\Delta ^{(1)}$ has symbolic representation $\delta
^{(1)}$ (with final letter in $E$). For each $i=1,\dots,\left\vert
E\right\vert$, choose a path $\lambda _{i}\in E_{a}$ linking $\delta ^{(1)}$
with $\chi _{i}$. Assume $\left\vert \lambda _{i}\right\vert =n_{i}$. The
definition of $\tau (\mu _{1})$ established in (\ref{Delta1}) implies that
for each $\varepsilon >0$ and all $k\geq k_{0}(\varepsilon )$ we have 
\begin{eqnarray*}
\sum_{|\sigma |=k-1,\sigma \in E_{a}}\left\Vert T(\chi _{i},\sigma
)\right\Vert ^{q} &\leq &C_{i}\sum_{|\sigma |=k-1,\sigma \in
E_{a}}\left\Vert T(\delta ^{(1)},\lambda _{i},\chi _{i},\sigma )\right\Vert
^{q} \\
&\leq &C_{i}\sum_{|\sigma |=k+n_{i},\sigma \in E_{a}}\left\Vert T(\delta
^{(1)},\sigma )\right\Vert ^{q}\leq C_{i}^{\prime }\rho _{\min}^{k(\tau (\mu
_{1},q)-\varepsilon )}.
\end{eqnarray*}
Consequently, $\tau _{E}(\mu ,q)\geq \tau (\mu _{1},q)$.

It is easy to see that $\tau (\mu _{1},q)\geq \tau (\mu _{1}+\cdot \cdot
\cdot +\mu _{J},q)$ and hence $\tau _{E}(\mu ,q)=\tau (\mu _{1}+\cdot \cdot
\cdot +\mu _{J},q)$ for all $q$.

Finally, we note that in \cite[Lemma 5.3]{FengSmoothness}, Feng shows (in
the equicontractive case, but the same argument works in general) that for $%
q\geq 0$, $\tau (\mu ,q)=\tau (\mu _{1},q)$. Hence $\tau (\mu ,q)=\tau
_{E}(\mu ,q)$ for all $q\geq 0$.
\end{proof}

\begin{corollary}
\label{Cor:min}For $q\geq 0$, the function $\tau _{L}(\mu ,q)$ is minimized
over all maximal loop classes $L$ at $L$ equal to the essential class $E$.
Likewise, the minimum value of $d_{\min}^{L}$ is attained at $L=E$.
\end{corollary}

\begin{proof}
The previous work shows that for all $q\geq 0$, $\tau _{L}(\mu ,q)\geq \tau
(\mu ,q)=\tau _{E}(\mu ,q)$ and that $\tau _{L}(\mu ,q)/q\rightarrow d_{\min
}^{L}$ as $q\rightarrow \infty$ for any loop class $L$.
\end{proof}

\section{The $L^{q}$-spectrum of finite type measures}

\label{S:Main}

\subsection{Main Theorem}

Throughout this section we continue to assume $\mu$ is a self-similar
measure of finite type. We conjecture that for all real $q$, that $\tau (\mu
,q)=\min_{L}\tau _{L}(\mu ,q)$, where the minimum is taken over all the
maximal loop classes $L$. Indeed, we have already seen in Theorem \ref%
{tauletauL} that $\tau (q)\leq \min_{L}\tau _{L}(q)$.

In this section, we will prove the conjecture under additional assumptions,
which we will see later are satisfied by many examples, including the
(uniform or biased) Bernoulli convolutions with contraction factor the
inverse of a simple Pisot number and the $(m,d)$-Cantor measures with $m\geq
d$.

It is convenient to introduce further terminology: We will say that two
maximal loop classes, $L_{1}$ and $L_{2}$, are adjacent if there is a path
from $L_{1}$ to $L_{2}$ that does not pass through any other maximal loop
class. To be more precise, we mean there is a path $\chi _{1},\dots,\chi
_{s} $ with $\chi _{1}\in L_{1}, \chi _{s}\in L_{2}$ and $\chi _{i}$ not in
a maximal loop class for any $i\neq 1,s$. Any such path will be called a 
\textbf{transition path} from $L_{1}$ to $L_{2}$. The fact that there are
only finitely many characteristic vectors ensures there are only finitely
many transition paths.

Recall that we write $E$ for the essential class.

\begin{theorem}
\label{main}Let $\mu$ be any self-similar measure of finite type. We have 
\begin{equation*}
\tau (\mu ,q)=\min_{L}\tau _{L}(\mu ,q)=\tau _{E}(\mu ,q)\text{ for }q\geq 0,
\end{equation*}
where the minimum is over all maximal loop classes $L$ (including $E$) and $%
E $ denotes the essential class.

If the transition matrices of all transition paths are positive matrices,
then 
\begin{equation*}
\tau (\mu ,q)=\min_{L}\tau _{L}(\mu ,q)\text{ for }q<0.
\end{equation*}
\end{theorem}

\begin{proof}
In Theorem \ref{tauletauL} and Corollary \ref%
{Cor:min}, we have already seen that $\tau (q)=\min_{L}\tau _{L}(q)=\tau
_{E}(q)$ for $q\geq 0$ and that $\tau (q)\leq \min_{L}\tau _{L}(q)$ for all $%
q$. Thus we need only prove the reverse inequality for $q<0~$under the
additional assumption, and that requires finding upper bounds on $%
\sum_{\left\vert \sigma \right\vert =k, \sigma \in \Omega^0}\left\Vert T(\sigma )\right\Vert ^{q}$%
. To do this, we observe that each such $\sigma $ can be factored as $(\beta
_{1}^{-},\lambda _{1}^{-},\beta _{2}^{-},\lambda _{2}^{-},\dots ,\lambda
_{\ell }^{-},\beta _{\ell +1})$ where

\begin{itemize}
\item $\beta _{1}=\beta _{1}^{(1)}\dots \beta _{1}^{(j_{1})}$ is the path
from $\gamma _{0}$ to the first maximal loop class. Here $\beta
_{1}^{(j_{1})}$ is in the (first) maximal loop class associate with $\lambda
_{1}$. This may be a singleton if $\gamma _{0}$ is within the loop class
associated with $\lambda _{1}$.

\item $\lambda_i = \lambda_i^{(1)} \dots \lambda_i^{(k_i)}$ is path
contained within a single loop class.

\item For $i=2,...,\ell $, $\beta _{i}=\beta _{i}^{(1)}\dots \beta
_{i}^{(j_{i})}$ is a transition path from the loop class given by $\lambda
_{i-1}$ to the loop class given by $\lambda _{i}$. Here $\beta _{i}^{(1)}$
is in the loop class associated with $\lambda _{i-1}$ and $\beta
_{i}^{(j_{i})}$ is in the loop class associated with $\lambda _{i}$.

\item $\beta_{\ell+1} = \beta_{\ell+1}^{(1)} \dots
\beta_{\ell_1}^{(j_{\ell+1})}$ is a prefix of a transition path from the
loop class given by $\lambda_\ell$. This may be a singleton if the last
characteristic vector in $\sigma$ is in a loop class.

\item We have $\beta_i^{(j_i)} = \lambda_{i}^{(1)}$ and $\lambda_i^{(k_i)} =
\beta_{i+1}^{(1)}$.
\end{itemize}

With this notation we have 
\begin{equation*}
T(\sigma )=T(\beta _{1})T(\lambda _{1})\dots T(\beta _{\ell })T(\lambda
_{\ell })T(\beta _{\ell +1}).
\end{equation*}%
Here we understand $T(\beta _{i})$ or $T(\lambda _{i})$ to be 1 if they
contained only one characteristic vector.

We notice that if there is a transition path between two maximal loop
classes, then there is no path (transition or otherwise) going the other
direction between these loop classes. Hence, as there are only a finite
number of maximal loop classes we see that $\ell$ is bounded above by the
number of maximal loop classes.

We notice that there are only a finite number of initial and transition
paths, and hence a finite number of $\beta _{i}$. We further note that, by
assumption, $T(\beta _{i})$ is a positive matrix for $i=2,\dots ,\ell $. We
always have that $T(\beta _{1})$ is a positive matrix with row dimension 1.
Hence by Lemma \ref{useful} (ii) we get 
\begin{align}
\Vert T(\sigma )\Vert & \geq \Vert T(\beta _{1})T(\lambda
_{1})\dots T(\beta _{\ell })T(\lambda _{\ell })T(\beta _{\ell +1})\Vert 
\notag \\
& \geq C\Vert T(\lambda _{1})\Vert \Vert T(\lambda _{2})\Vert \dots \Vert
T(\lambda _{\ell -1})\Vert \Vert T(\lambda _{\ell })T(\beta _{\ell +1})\Vert
\label{eq:partial decomp}
\end{align}

Recall $\beta _{\ell +1}=\beta _{\ell +1}^{(1)}\dots \beta _{\ell
+{1}}^{(j_{\ell +1})}$ is a prefix of a transition path. If $\beta _{\ell
+1} $ is a singleton, then we obtain directly that $\Vert T(\lambda _{\ell
})T(\beta _{\ell +1})\Vert =\Vert T(\lambda _{\ell })\Vert $. Assume that it
is not a singleton. Let $\beta _{\ell +1}^{+}=\beta _{\ell +1}^{(1)}\dots
\beta _{\ell _{1}}^{(j_{\ell +1})}\dots \beta _{\ell _{1}}^{(J_{\ell +1})}$
be a transition path for which it is a prefix. We see that 
\begin{align*}
\Vert T(\lambda _{\ell })T(\beta _{\ell +1})\Vert \Vert T(\beta _{\ell
+1}^{(j_{\ell +1})}\dots \beta _{\ell +1}^{(J_{\ell +1})})\Vert & \geq \Vert
T(\lambda _{\ell })T(\beta _{\ell +1})T(\beta _{\ell +1}^{(j_{\ell
+1})}\dots \beta _{\ell +1}^{(J_{\ell +1})})\Vert \\
& \geq \Vert T(\lambda _{\ell })T(\beta _{\ell +1}^{+})\Vert \\
& \geq b\Vert T(\lambda _{\ell })\Vert
\end{align*}%
There are only finitely many such completions, hence $\Vert T(\beta _{\ell
+1}^{(j_{\ell +1})}\dots \beta _{\ell +1}^{(J_{\ell +1})}\Vert $ is bounded
from below. Combining this observation with equation 
\eqref{eq:partial decomp} 
gives us that $\Vert T(\sigma )\Vert \geq C\prod_{i=1}^{\ell }\Vert
T(\lambda _{i})\Vert $ for some (new) constant $C$, or equivalently 
\begin{equation}
\Vert T(\sigma )\Vert ^{q}\leq C^{q}\prod_{i=1}^{\ell }\Vert T(\lambda
_{i})\Vert ^{q}.  \label{C}
\end{equation}

Let $\kappa $ be the number of possible tuples $(\beta _{1},\beta _{2},\dots
,\beta _{\ell +1})$. We see that $\kappa $ is finite as there are a finite
number of transition paths and $\ell $ is bounded above. 
Let $C$ be a renamed constant, the minimum value of $C^{q}$ taken over all
such tuples. 
Let $L_1, \dots, L_N$ be the complete set of loop classes (including the 
    essential class).

With this notation we have 
\begin{align*}
\sum_{|\sigma |=n}\Vert T(\sigma )\Vert ^{q}
    & \leq \sum_{(\beta _{1},\dots,\beta _{\ell +1})}\left( \sum_{|\lambda _{1}|+\dots +|\lambda _{\ell }|\leq
n}\Vert T(\beta _{1})T(\lambda _{1})\dots T(\lambda _{\ell })T(\beta
_{\ell +1})\Vert ^{q}\right) \\
    & \leq C\sum_{(\beta _{1},\dots ,\beta _{\ell +1})} \left(
    \sum_{i_1 + \dots + i_N \leq n} 
    \sum_{\substack{|\lambda_1| = i_1\\ \lambda_1 \in L_1}}  \dots
    \sum_{\substack{|\lambda_N| = i_N\\ \lambda_N \in L_N}} 
    \Vert T(\lambda _{1})\Vert ^{q}\dots \Vert T(\lambda _{N})\Vert^{q}\right)\\
    & \leq C \kappa
    \sum_{i_1 + \dots + i_N \leq n} 
   \left(\sum_{\substack{|\lambda_1|=i_1 \\ \lambda_1\in L_1}}\Vert T(\lambda_1)\Vert^q\right)
   \dots
   \left(\sum_{\substack{|\lambda_N|=i_N \\ \lambda_N\in L_N}}\Vert T(\lambda_N)\Vert^q\right)
\end{align*}

Fix $\varepsilon >0$ and let 
\begin{equation*}
\theta (q):=\min_{L}\tau _{L}(q).
\end{equation*}%
By the definition of $\tau _{L}$, there is some $n_{L}$ such that for all
paths in $L$ with length at least $n\geq n_{L}$, 
\begin{equation*}
\sum_{\lambda \in L_{a},\left\vert \lambda \right\vert =n}\left\Vert
T(\lambda )\right\Vert ^{q}\leq \rho _{\min }^{n(\tau _{L}(q)-\varepsilon
)}\leq \rho _{\min }^{n(\theta (q)-\varepsilon )}.
\end{equation*}%
As there are only finitely many paths of length less than $n_{L}$ and only
finitely many maximal loop classes, there is some constant $C_{\varepsilon }$
such that 
\begin{equation*}
\sum_{\lambda \in L_{a},\left\vert \lambda \right\vert =n}\left\Vert
T(\lambda )\right\Vert ^{q}\leq C_{\varepsilon }\rho _{\min }^{n(\theta
(q)-\varepsilon )}
\end{equation*}%
for all $n$ and for all maximal loop classes $L$. Thus 
\begin{align*}
\sum_{|\sigma|=n }\Vert T(\sigma )\Vert ^{q}
    & \leq C \kappa
    \sum_{i_1 + \dots + i_N \leq n} 
   \left(\sum_{\substack{|\lambda_1|=i_1 \\ \lambda_1\in L_1}}\Vert T(\lambda_1)\Vert^q\right)
   \dots
 \left(\sum_{\substack{|\lambda_N|=i_N\\ \lambda_N\in L_N}}\Vert T(\lambda_N)\Vert^q\right)\\
    & \leq C \kappa
    \sum_{i_1 + \dots + i_N \leq n} 
    \left( C_\epsilon \rho_{\min}^{i_1 (\theta(q)-\epsilon)} \right)
    \left( C_\epsilon \rho_{\min}^{i_2 (\theta(q)-\epsilon)} \right) \dots
    \left( C_\epsilon \rho_{\min}^{i_N (\theta(q)-\epsilon)} \right) \\
    & \leq C \kappa C_\epsilon^N \sum_{i_1 + \dots + i_N \leq n} 
      \rho_{\min}^{n (\theta(q)-\epsilon)} \\
    & \leq  C \kappa C_\epsilon^N n^N 
      \rho_{\min}^{n (\theta(q)-\epsilon)} \\
    & \leq  C' \rho_{\min}^{n (\theta(q)-2\epsilon)} 
\end{align*}
Recall here that $\theta(q) < 0$ when $q < 0$.

Taking the logarithm, dividing by $n\log \rho _{\min }$ and letting $%
n\rightarrow \infty $, we deduce that 
\begin{equation*}
\tau (q)=\liminf_{n}\frac{\log \sum_{\left\vert \sigma \right\vert
=n}\left\Vert T(\sigma )\right\Vert ^{q}}{n\log \rho _{\min }}\geq \theta
(q)-\varepsilon \text{.}
\end{equation*}%
As this holds for all $\varepsilon >0$, we have $\tau (q)\geq \theta
(q)=\min_{L}\tau _{L}(q)$, as we desired to show.
\end{proof}

\begin{remark}
We remark that while the values of the entries of the transition matrices
associated with transition paths are a function of the probabilities
associated with the self-similar measure, the property of being positive
matrices depends only on the finite type structure of the IFS. Thus Theorem %
\ref{main} will apply to every self-similar measure associated with an IFS
which satisfies the hypotheses.
\end{remark}

\subsection{Consequences of the Theorem}

For this subsection, we will let 
\begin{equation*}
d=\max \{d_{\max}^{L}:L\text{ maximal loop class, }L\neq E\}.
\end{equation*}

\begin{proposition}
\label{summary}Let $\mu$ be any self-similar measure of finite type. Suppose
that every maximal loop class other than the essential class is a singleton
and that the primitive transition matrices associated with these singleton
loop classes are scalars. Then $\tau (\mu ,q)=\min_{L}\tau _{L}(\mu ,q)$ for
all $q$ and consequently, 
\begin{equation*}
\tau (\mu ,q)=\left\{ 
\begin{array}{cc}
\tau _{E}(\mu ,q) & \text{if }q\geq 0 \\ 
\min \{qd,\tau _{E}(\mu ,q)\} & \text{if }q<0%
\end{array}
\right. .
\end{equation*}
\end{proposition}

\begin{proof}
If $L$ is a singleton maximal loop class with primitive transition matrix a
scalar, then the transition matrix of any transition path beginning at $L$
is a row vector. As every column of a transition matrix has a non-zero
entry, it follows that such a transition matrix is positive.

There are no transition paths starting from the essential class. As all
other maximal loop classes are singletons with scalar transition matrices,
we have that all transition paths are positive. Thus the hypotheses of
Theorem \ref{main} are satisfied and $\tau (\mu ,q)=\min_{L}\tau _{L}(\mu ,q)
$ for all $q$. 

The remaining claims follow upon recalling that if $L$ is a singleton loop
class, then $\tau _{L}(\mu ,q)=qd_{\max }^{L}$ (and $=qd_{\min }^{L}$) for
all $q$; see Corollary \ref{box} \ref{box iii}. 
\end{proof}

\begin{remark}
As with Theorem \ref{main}, the assumptions of this Proposition are
structural properties of the underlying IFS and hence every self-similar
measure associated with such an IFS will satisfy the Proposition.
\end{remark}

For both the Bernoulli convolutions with simple Pisot inverses as
contractions and the $(m,d)$-Cantor measures, there are three maximal loop
classes, the essential class and two singletons (corresponding to the
endpoints $0,1$) and the latter have scalars as their primitive transition
matrices. Hence the Proposition applies to these examples.\ We will discuss
these classes of finite type measures in more detail later in this
subsection; see Examples \ref{Cantor} and \ref{Pisot}.

In fact, these Bernoulli convolutions and Cantor-like measures are special
examples of the following more general situation.

\begin{lemma}
\label{structure}Suppose the IFS is equicontractive and of finite type and
that the set of essential points is $(0,1)$. Then there are (precisely) two
other maximal loop classes. These are both singletons, corresponding to the
points $0,1$ respectively, with primitive transition matrices that are
scalars.
\end{lemma}

\begin{proof}
Assume the IFS consists of similarities $\{S_{j}\}_{j=0}^{m}$, ordered in
the natural way. The left-most net interval at level $1$ is either $%
[0,S_{0}(1)]$ or $[0,S_{1}(0)]$, depending on which is smaller, $S_{0}(1)$
or $S_{1}(0)$. At level $2$, it is the interval $[0,a]$ where $a=\min
\{S_{i}S_{j}(0),S_{i}S_{j}(1):i,j\}$. It is straight forward to check that
if $S_{0}(1)\leq S_{1}(0)$, then the minimum value is $S_{0}S_{0}(1)$ and
then the net interval $[0,a]$ has normalized length $1$, the same as the
level $1$ net interval $[0,S_{0}(1)]$. Otherwise, the minimum value is $%
S_{0}S_{1}(0)$ and the net interval $[0,a]$ has the same normalized length
as the level $1$ net interval, $[0,S_{1}(0)]$. In both cases, these level
one and level two net intervals have the same neighbour set, the singleton $%
\{(0,0)\}$. Hence if $\chi _{0}$ is the characteristic vector of the
left-most interval at level one, $\chi _{0}$ is also its left-most child, so 
$\{\chi _{0}\}$ is a loop class and $0$ has symbolic representation $(\gamma
_{0},\chi _{0},\chi _{0},\dots )$. As there is only one neighbour set, the
primitive transition matrix is a scalar. As $0$ is not an essential point, $%
\chi _{0}$ is not in the essential class.

By symmetry the same is true for the right-most net intervals with $1$
having symbolic representation $(\gamma _{0},\chi _{1},\chi _{1},\dots)$.

The loop classes $L_{0}=\{\chi _{0}\}$ and $L_{1}=\{\chi _{1}\}$ are maximal
because otherwise there would be infinitely many points that are not
essential points. There are no other maximal loop classes (other than the
essential class) for the same reason.
\end{proof}

\begin{corollary}
\label{E}Suppose $\mu$ is an equicontractive, finite type measure and that
the set of essential points is $(0,1)$. Let $\Delta$ be any finite union of
net intervals of a fixed level, with characteristic vectors in the essential
class $E$.

\begin{enumerate}
\item Then 
\begin{equation*}
\tau (\mu ,q)=\left\{ 
\begin{array}{cc}
\tau _{E}(\mu ,q)=\tau (\mu |_{\Delta },q) & \text{if }q\geq 0 \\ 
\min \{qd,\tau _{E}(\mu ,q)\} & \text{if }q<0%
\end{array}
\right. .
\end{equation*}
\label{E i}

\item If $d_{\max}^{E}<d$, then there is a choice of $q_{0}<0$ such that 
\begin{equation}
\tau (\mu ,q)=\left\{ 
\begin{array}{cc}
\tau _{E}(\mu ,q)=\tau (\mu |_{\Delta },q) & \text{if }q\geq q_{0} \\ 
qd & \text{if }q<q_{0}%
\end{array}
\right. .  \label{C1}
\end{equation}
We can take $q_{0}\geq -1/(d-d_{\max}^{E})$ to be the unique solution to $%
\tau _{E}(q)=qd$. \label{E ii}

\item If $d_{\max}^{E}\geq d$, then $\tau (\mu ,q)=\tau _{E}(\mu ,q)=\tau
(\mu |_{\Delta },q)$ for all $q$. \label{E iii}
\end{enumerate}
\end{corollary}

\begin{proof}
Part \ref{E i}. This follows directly from the previous work.

For Part \ref{E ii}, recall that we saw in Proposition \ref{taubound} that
if $q<0$, then 
\begin{equation*}
qd_{\max }^{E}\geq \tau _{E}(q)\geq qd_{\max }^{E}-1.
\end{equation*}%
Thus if $q_{1}=-1/(d-d_{\max }^{E})$, then $\tau _{E}(q_{1})\geq dq_{1}$. It
is easy to see that $\tau _{E}(0)<0$. As $\tau _{E}$ is concave, and hence
continuous, there is some $q_{0}\in \lbrack q_{1},0)$ such that $\tau
_{E}(q_{0})=dq_{0}$. Using the fact that if $f$ is a concave function and $%
z\in (x_{1},x_{2})$ we have 
\begin{equation*}
\frac{f(x_{1})-f(z)}{z-x_{1}}\leq \frac{f(z)-f(x_{2})}{x_{2}-z},
\end{equation*}%
it can be checked that there can only be one choice of $q_{0}$ with $\tau
_{E}(q_{0})=dq_{0}$. The statements of (\ref{C1}) are clearly satisfied for
this choice of $q_{0}$ $\geq q_{1}=-1/(d-d_{0})$.

Part \ref{E iii}. We simply note that if $d_{\max}^{E}\geq d$ and $q<0$,
then $\tau _{E}(q)\leq qd_{\max}^{E}\leq qd$ and hence $\min \{qd,\tau
_{E}(\mu ,q)\}=\tau _{E}(\mu ,q)$.
\end{proof}

\begin{corollary}
Assume $\mu$ is associated with an equicontractive, finite type IFS. Suppose 
$K_{E}=$ $(0,1)$ and that 
\begin{equation*}
\sup \{\dim _{\mathrm{loc}}\mu (x):x\in K_{E}\}<\max \{\dim _{\mathrm{loc}%
}\mu (0),\dim _{\mathrm{loc}}\mu (1)\}.
\end{equation*}
Then 
\begin{equation*}
d=\max_{j=0,1}\dim _{\mathrm{loc}}\mu (j)>d_{\max}^{E}
\end{equation*}
If $\Delta$ is any non-empty, closed subinterval of $(0,1)$, then 
\begin{equation*}
\tau (\mu ,q)=\left\{ 
\begin{array}{cc}
\tau (\mu |_{\Delta },q) & \text{if }q\geq q_{0} \\ 
qd & \text{if }q<q_{0}%
\end{array}
\right.
\end{equation*}
\newline
where $q_{0}<0$ is the unique solution to $\tau (\mu |_{\Delta },q)=qd$ .
\end{corollary}

\begin{proof}
As $0,1$ are the endpoints of the support of $\mu $, they each have a unique
periodic symbolic representation. Hence $\dim _{\mathrm{loc}}\mu (j)=$ $%
d_{\max }^{L_{j}}$ where $L_{j}$ is the singleton maximal loop class
associated with $j=0,1$. Thus our hypothesis, together with Corollary \ref%
{Cor:dE}, implies $d_{\max }^{E}<d$ and therefore Corollary \ref{E} \ref{E
ii} applies.

It is easy to see from the definition of the $L^{q}$-spectrum that if $%
Y\subseteq X\subseteq $supp$\mu$, then $\tau (\mu |_{X},q)\leq \tau (\mu
|_{Y},q)$ for all $q$. Since any $\Delta$ described in the statement of the
Corollary will contain a net interval with a characteristic vector in the
essential class and is contained in a finite union of such net intervals, it
follows that $\tau _{E}(\mu ,q)=\tau (\mu |_{\Delta },q)$ for $q\geq q_{0}$.
\end{proof}

\begin{example}[$(m,d)$-Cantor measures]
\label{Cantor} Consider the equicontractive, finite type IFS with
contractions $S_{j}(x)=x/3+2j/9$ for $j=0,1,2,3$. Any associated
self-similar measure is a $(3,3)$-Cantor measure. There are two singleton
maximal loop classes, $L_{0}=\{\gamma _{1}\}$ and $L_{1}=\{\gamma _{7}\}$,
corresponding to the endpoints $0,1$. The characteristic vectors $\{\gamma
_{2},\gamma _{3},\gamma _{5},\gamma _{6}\}$ comprise the essential class $E$%
, with the open interval $(0,1)$ being the set of essential points. See
Figure \ref{fig:Cantor} for the transition graph. Thus Corollary \ref{E} \ref%
{E i} applies for any choice of probabilities.

If we take, for example, the probabilities $1/3,3/8,3/8,1/8$, the associated
measure is the rescaled $3$-fold convolution of the classical Cantor
measure. We have 
\begin{equation*}
d=\dim _{\mathrm{loc}}\mu (j)=\frac{\log 8}{\log 3}=d_{\max}^{L_{j}}\text{
for } j=0,1.
\end{equation*}
It is known that $\sup_{x\in K_{E}}\{\dim _{\mathrm{loc}}\mu (x)\}<\log
8/\log 3=d$ and that the union of the net intervals of level $N$ with
characteristic vectors in $E$ is the closed interval $\Delta (N)=[2\cdot
3^{-N},1-2\cdot 3^{-N}]$ .

Similar statements are true for any choice of regular probabilities, hence
we can appeal to Corollary \ref{E} \ref{E ii} in the regular case to
conclude that for any $N$, 
\begin{equation*}
\tau (\mu ,q)=\left\{ 
\begin{array}{cc}
\tau (\mu |_{\Delta (N)},q) & \text{if }q\geq q_{0} \\ 
qd & \text{if }q<q_{0}%
\end{array}
\right.
\end{equation*}
where $q_{0}<0$ is the unique solution to $\tau (\mu |_{\Delta (1)},q)=qd$.

More generally, similar statements hold for any $(m,d)$-Cantor measure with $%
m\geq d\geq 2$, with, again, the stronger conclusions if regular
probabilities are chosen. In this case, the union of the net intervals of
level $N$ with characteristic vectors in $E$ is the closed interval 
\begin{equation*}
\Delta (N)=[(d-1)/(md^{N}),1-(d-1)/(md^{N})].
\end{equation*}

Most of these results were previously found for the special cases of the $3$
-fold convolution of the Cantor measure in \cite[Thm. 1.1, 1.3]{FLW} and for 
$m<2d-2$ in \cite[Thm. 1.5]{Sh}. The reader can refer to \cite[Section 7]%
{HHM} and \cite[Section 5]{HHN} for proofs of the finite type structural
properties of the $(m,d)$-Cantor measures and facts about their local
dimensions. \ 

\begin{figure}[tbp]
\label{F:Cantor} 
\begin{tikzpicture}[node distance=1.5cm]
\node[circle, draw] (cr1) {$\gamma_0$};
\node[below of = cr1] (dummyB) {};
\node[circle, draw, below of = cr1] (cr3c) {$\gamma_4$};
\node[left of = dummyB] (dummyA) {};
\node[right of = dummyB] (dummyC) {};
\node[circle, draw, left of = dummyA] (cr2) {$\gamma_1$};
\node[circle, draw, right of = dummyC] (cr5) {$\gamma_7$};
\node[below of = cr2] (dummyD) {};
\node[below of = dummyA] (dummyE) {};
\node[below of = dummyB] (dummy3c) {};
\node[below of = dummyC] (dummyF) {};
\node[below of = cr5] (dummyG) {};
\node[circle, draw, below of = dummyD] (cr3a) {$\gamma_2$};
\node[below of = dummyE] (dummyH) {};
\node[below of = dummyF] (dummyI) {};
\node[circle, draw, below of = dummyG] (cr4a) {$\gamma_6$};
\node[circle, draw, below of = dummyH] (cr3b) {$\gamma_3$};
\node[circle, draw, below of = dummyI] (cr4b) {$\gamma_5$};

\path
  (cr1) edge[-triangle 45] (cr2)
  (cr1) edge[-triangle 45] (cr3a)
  (cr1) edge[-triangle 45] (cr3b)
  (cr1) edge[-triangle 45] (cr3c)
  (cr1) edge[-triangle 45] (cr4a)
  (cr1) edge[-triangle 45] (cr4b)
  (cr1) edge[-triangle 45] (cr5)
  (cr2) edge[>=triangle 45, loop left] (cr2)
  (cr2) edge[-triangle 45] (cr3a)
  (cr2) edge[-triangle 45] (cr3b)
  (cr2) edge[-triangle 45] (cr4a)
  (cr2) edge[-triangle 45] (cr4b)
  (cr5) edge[>=triangle 45, loop right] (cr5)
  (cr5) edge[-triangle 45] (cr3a)
  (cr5) edge[-triangle 45] (cr3b)
  (cr5) edge[-triangle 45] (cr4a)
  (cr5) edge[-triangle 45] (cr4b)
  (cr3a) edge[>=triangle 45, loop left] (cr3a)
  (cr3a) edge[-triangle 45, bend left=10] (cr4a)
  (cr3a) edge[-triangle 45, bend left=20] (cr3b)
  (cr3b) edge[>=triangle 45, loop left] (cr3b)
  (cr3b) edge[-triangle 45] (cr4a)
  (cr3b) edge[-triangle 45, bend left=20] (cr3a)
  (cr3c) edge[-triangle 45] (cr3b)
  (cr3c) edge[-triangle 45] (cr4a)
  (cr3c) edge[-triangle 45] (cr3a)
  (cr4a) edge[>=triangle 45, loop right] (cr4a)
  (cr4a) edge[-triangle 45, bend left=10] (cr3a)
  (cr4a) edge[-triangle 45, bend left=20] (cr4b)
  (cr4b) edge[>=triangle 45, loop right] (cr4b)
  (cr4b) edge[-triangle 45] (cr3a)
  (cr4b) edge[-triangle 45, bend left=20] (cr4a);
\end{tikzpicture}
\caption{$(3,3)$-Cantor measure}
\label{fig:Cantor}
\end{figure}
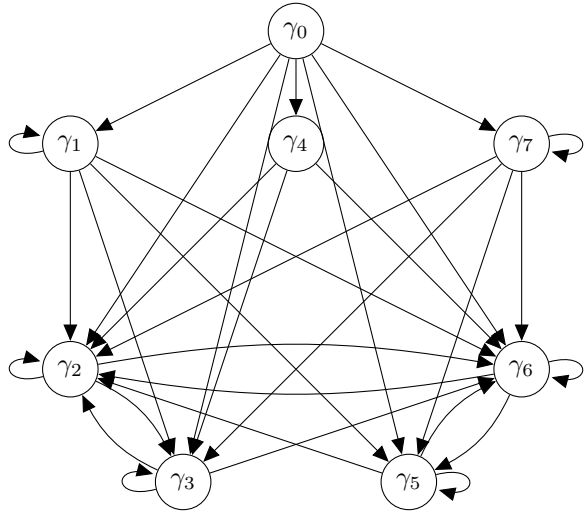
\end{example}

\begin{example}[Bernoulli convolutions]
\label{Pisot} Consider the IFS $S_{0}(x)=\varrho x,S_{1}(x)=\varrho
x+1-\varrho $ which generates the Bernoulli convolution with contraction
factor $\varrho $ the inverse of the golden mean. There are two singleton
maximal loop classes, $L_{0}=\{\gamma _{1}\}$ and $L_{1}=\{\gamma _{3}\}$,
corresponding to the endpoints $0,1$. The remaining maximal loop class given
by $\{\gamma _{2},\gamma _{4},\gamma _{5},\gamma _{6}\}$ form the essential
class $E$, so again the open interval $(0,1)$ is the set of essential
points. See Figure \ref{fig:Bernoulli}.

If $\mu $ is the Bernoulli convolution arising from the probabilities $%
p_{0},p_{1}=1-p_{0}$, then 
\begin{equation*}
d_{\max }^{L_{j}}=\dim _{\mathrm{loc}}\mu (j)=\frac{\log p_{j}}{\log \varrho 
}\text{ for }j=0,1.
\end{equation*}%
In the case that $p_{0}<p_{1}$ (the biased case), it was shown in \cite[Thm.
4.3]{HHN} that $d_{\max }^{E}<d=d_{\max }^{L_{0}}$. In contrast, it was
shown in \cite[Thm. 1.5]{FengLimited} that in the uniform case $d=d_{\max
}^{L_{j}}=d_{\max }^{E}$ for $j=0,1$.

In either the uniform or biased case we have 
\begin{equation*}
\tau (\mu ,q)=\left\{ 
\begin{array}{cc}
\tau _{E}(\mu ,q)=\tau (\mu |_{\Delta },q) & \text{if }q\geq 0 \\ 
\min \{qd,\text{ }\tau _{E}(\mu ,q)\} & \text{if }q<0%
\end{array}
\right. .
\end{equation*}
where $\Delta$ is any finite union of net intervals of a fixed level with
characteristic vectors in the essential class $E$. We can take as $\Delta$
any interval of the form $[\varrho ^{N}(1-\varrho ),1-\varrho ^{N}(1-\varrho
)]$, for example. In the biased case, $\tau (\mu ,q)$ is the straight line $%
qd$ for $q<q_{0}$, as per Corollary \ref{E} \ref{E ii}, while in the uniform
case, Corollary \ref{E} \ref{E iii} implies $\tau _{E}(\mu ,q)=\tau (\mu ,q)$
for all $q$.

In \cite[Thm. 4.12]{FengLimited}, Feng showed that in the uniform case $\tau
(\mu ,q)$ is the line $qd$ for $q\leq q_{1}$ for a suitable choice of $%
q_{1}<0$ and was not differentiable at that point. Hence $\tau _{E}$ is not
differentiable at $q_{1}$.

Similar statements hold for the Bernoulli convolutions with contraction
factor the inverse of a simple Pisot number other than the golden mean. (A
real number is called a simple Pisot number if it is the (unique) positive
root of a polynomial $x^{k}-x^{k-1}\cdot \cdot \cdot -x-1=0$ for some
integer $k\geq 2$.) In particular, in the biased case, it was shown in \cite[%
Section 4]{HHN} that $d>d_{\max}^{E}$, so we have $\tau (\mu ,q)= qd$ for $%
q<q_{0}$.

In the uniform case, again $d=d_{\max}^{E}$, so $\tau _{E}(\mu ,q)=\tau (\mu
,q)$ for all $q$. Moreover, Feng in \cite[Theorem 5.8]{FengLimited} proved
that $\tau (\mu ,q)=qd-\log x(q)/\log \varrho$ for an infinitely
differentiable function $x$ and that always $x(q)<1$. Thus, in contrast to
the golden mean case, here we have $\tau _{E}(\mu ,q)<qd$ for all $q$ and
the function $\tau _{E}$ is differentiable everywhere.

\begin{figure}[tbp]
\label{F:BC} 
\begin{tikzpicture}[node distance=1.5cm]
\node[circle, draw] (cr1) {$\gamma_0$};
\node[below of = cr1] (dummy) {};
\node[circle, draw, left of = dummy] (cr2){$\gamma_1$};
\node[circle, draw, right of = dummy] (cr4) {$\gamma_3$};
\node[circle, draw, below of = dummy] (cr3a) {$\gamma_2$};
\node[circle, draw, below of = cr3a] (cr5) {$\gamma_4$};
\node[circle, draw, left of = cr5] (cr3b) {$\gamma_6$};
\node[circle, draw, right of = cr5] (cr6) {$\gamma_5$};

\path
  (cr1) edge[-triangle 45] (cr2)
  (cr1) edge[-triangle 45] (cr3a)
  (cr1) edge[-triangle 45] (cr4)
  (cr2) edge[>=triangle 45, loop left] (cr2)
  (cr2) edge[-triangle 45] (cr3a)
  (cr3a) edge[-triangle 45, bend left = 20] (cr5)
  (cr3b) edge[-triangle 45, bend left = 20] (cr5)
  (cr4) edge[>=triangle 45, loop right] (cr4)
  (cr4) edge[-triangle 45] (cr3a)
  (cr5) edge[-triangle 45, bend left = 20] (cr3a)
  (cr5) edge[-triangle 45, bend left = 20] (cr3b)
  (cr5) edge[-triangle 45] (cr6)
  (cr6) edge[-triangle 45] (cr3a);
\end{tikzpicture}
\caption{Bernoulli convolution}
\label{fig:Bernoulli}
\end{figure}
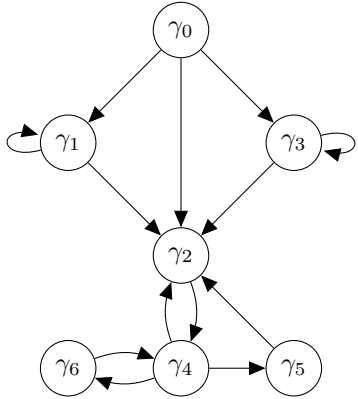
\end{example}

Here is another example which satisfies the conditions of Proposition \ref%
{summary}, but the set of essential points is not the full open interval $%
(0,1)$.

\begin{example}
Consider the equicontractive IFS with $S_{j}(x)=x/3+d_{j}$ for $%
d_{j}=0,1/5,2/5,2/3$. One can verify that there are three maximal loop
classes other than the essential class. These are all singletons
corresponding to the points $x=0,1$ and $x_{0}\in (0,1)$ respectively, (so $%
K_{E}\neq (0,1)$) and all have scalar primitive transition matrices.
Consequently, Theorem \ref{main} applies and thus $\tau (\mu ,q)=\min \tau
_{L}(\mu ,q)$ for all $q$ and all associated self-similar measures.

As a particular example, take $p_{3}=1/156 , p_{0}=5/256, p_{2}=25/156$ and $%
p_{1}=125/156$. Although we have not been able to prove this, computational
work suggests that $d_{\max}^{E}\leq 2.25$ and $d\geq 4.5$. If this is true,
then the $L^{q}$-spectrum would be the line $y=qd$ for large negative $q$.
\end{example}

For our final example, we study the $L^{q}$-spectrum of the self-similar
measure\ from Example \ref{ExSp1}.

\begin{example}
\label{ExSpecial} Take the IFS with $S_{j}(x)=x/3+d_{j}$ for $%
d_{j}=0,1/9,1/3,1/2,2/3$ and probabilities $p_{j}=4/17$ for $j=0,1,3,4$ and $%
p_{2}=1/17$, introduced in Example \ref{ExSp1}. We refer the reader to
Figure \ref{fig:ExSpecial} for the transition graph. There are 19
characteristic vectors and four singleton maximal loop classes, $%
L_{i}=\{\gamma _{i}\}$ for $i=1,4,7,8$, in addition to the essential class $E
$ with 9 characteristic vectors (whose transitions are not detailed in the
figure). The maximal loop classes $L_{1}$ and $L_{7}$ correspond to the
points at $0$ and $1$. These points only have one representation, hence
these loop classes are not adjacent to any other loop class. The loop class
at $L_{4}$ corresponds to a right most path, whose corresponding left most
path is in the essential class. The loop class at $L_{8}$ also corresponds
to a right most path, whose corresponding left most path is in the essential
class. The singleton maximal loop classes $L_{1},L_{4}$ (corresponding to
the point $1/2$) and $L_{7}$ have scalar primitive transition matrices, so
the transition matrices of transition paths beginning with these loop
classes are all positive. However, the maximal loop class $L_{8}$ has a $%
2\times 2$ lower triangular, primitive transition matrix. Furthermore, there
are 10 transition paths between $L_{8}$ and the essential class, namely 
\begin{equation*}
\begin{array}{llll}
\gamma _{8}\gamma _{3}\gamma _{12} & \gamma _{8}\gamma _{3}\gamma _{13} & 
\gamma _{8}\gamma _{3}\gamma _{14}\gamma _{5}\gamma _{9} & \gamma _{8}\gamma
_{3}\gamma _{14}\gamma _{5}\gamma _{10} \\ 
\gamma _{8}\gamma _{3}\gamma _{14}\gamma _{5}\gamma _{11} & \gamma
_{8}\gamma _{3}\gamma _{14}\gamma _{5}\gamma _{12} & \gamma _{8}\gamma
_{3}\gamma _{14}\gamma _{5}\gamma _{13} & \gamma _{8}\gamma _{3}\gamma
_{14}\gamma _{9} \\ 
\gamma _{8}\gamma _{3}\gamma _{14}\gamma _{10} & \gamma _{8}\gamma
_{3}\gamma _{14}\gamma _{11} &  & 
\end{array}%
\end{equation*}%
Some of the transition matrices for these transition paths (such as $\gamma
_{8}\gamma _{3}\gamma _{12}$) are not positive. So Theorem \ref{main} does
not apply directly.

However, the transition matrix $T(\gamma _{8}\gamma _{8}\gamma _{3})$ is
positive and one can only exit $L_{8}$ through $\gamma _{3}$. Furthermore,
one can only enter $L_{8}$ from $\gamma _{2}$. With this in mind, we can
modify our definitions slightly so that the proof of the Theorem will still
apply.

\begin{enumerate}
\item Include in the list of initial paths the paths starting with $\gamma
_{0}\gamma _{2}\gamma _{8}\gamma _{3}\dots $.

\item Include in the list of transition paths from $L_1$ to $E$ the paths
starting with $\gamma_1 \gamma_2 \gamma_8 \gamma_3 \dots$.

\item Require transition paths from $L_{8}$ to $E$ to start with $\gamma
_{8}\gamma _{8}\gamma _{3}\dots $.

\item Decompose $\sigma$ using these modified transition paths.
\end{enumerate}

With these modifications, there are still only finitely many initial and
transition paths and these are now all positive. There are also still only
finitely many choices for loop class components $\lambda _{i}$. From here
the Theorem follows as before and we again deduce that $\tau (\mu
,q)=\min_{L}\tau _{L}(\mu ,q)$ for all $q$.

We have $\tau _{L_{j}}(q)=qd_{\max }^{L_{j}}$ with $d_{\max }^{L_{j}}=\log
(17/4)/\log 3$ for $j=1,4,8$ and $d_{\max }^{L_{j}}=$ $\log 17/\log 3$ for $%
j=7$. Computational work, as explained in Appendix \ref{app}, shows that 
\begin{equation*}
\max \{\dim _{\mathrm{loc}}\mu (x):x\in K\}\leq \log (17/2)/\log 3.
\end{equation*}%
In particular, $d_{\max }^{E}<d_{\max }^{\Omega }=\log 17/\log 3=d$ and thus
for a suitable choice of $q_{0}$, $\tau (\mu ,q)$ is the straight line $\tau
(q)=qd_{\max }^{\Omega }$ for $q\leq q_{0}$. But $\{x:\dim _{\mathrm{loc}%
}\mu (x)=d_{\max }^{\Omega }\}$ is empty, contrary to the spirit of the
multifractal formalism.

\begin{figure}[tbp]
\label{F:ExSpec} 
\begin{tikzpicture}[node distance=1.5cm]

\node[circle, draw] (cr1) {$\gamma_0$};
\node[below of = cr1] (dummyA) {};
\node[left of = dummyA] (dummyB) {};
\node[circle, draw, left of = dummyB] (cr5) {$\gamma_4$};
\node[right of = dummyA] (dummyF) {};
\node[circle, draw, right of = dummyF] (cr2) {$\gamma_1$};
\node[circle, draw, below of = dummyB] (cr8) {$\gamma_8$};
\node[circle, draw, below of = dummyA] (cr3) {$\gamma_2$};
\node[below of = dummyF] (dummyC) {};
\node[circle, draw, right of = dummyC] (cr7) {$\gamma_7$};
\node[circle, draw, below of = cr8] (cr14) {$\gamma_{14}$};
\node[circle, draw, left of = cr14] (cr4) {$\gamma_3$};
\node[right of = cr14] (dummyD) {};
\node[circle, draw, right of = dummyD] (cr6a) {$\gamma_5$};
\node[circle, draw, right of = cr6a] (cr6b) {$\gamma_6$};
\node[below of = dummyD] (dummyE) {};
\node[circle, draw, below of = dummyE] (ec) {$
\substack{\mathrm{Essential Class} \\ \gamma_9, \gamma_{10}, \gamma_{11}, \gamma_{12}, \gamma_{13}\\ \gamma_{15}, \gamma_{16}, \gamma_{17}, \gamma_{18}}$};

\path
  (cr1) edge[-triangle 45] (cr2)
  (cr1) edge[-triangle 45] (cr3)
  (cr1) edge[-triangle 45, bend right=20] (cr4)
  (cr1) edge[-triangle 45] (cr5)
  (cr1) edge[-triangle 45] (cr6a)
  (cr1) edge[-triangle 45] (cr6b)
  (cr1) edge[-triangle 45] (cr7)
  (cr2) edge[>=triangle 45, loop right] (cr2)
  (cr2) edge[-triangle 45] (cr3)
  (cr3) edge[-triangle 45] (cr4)
  (cr3) edge[-triangle 45] (cr8)
  (cr3) edge[-triangle 45] (ec)
  (cr3) edge[-triangle 45] (cr6a)
  (cr4) edge[-triangle 45] (ec)
  (cr4) edge[-triangle 45] (cr14)
  (cr5) edge[>=triangle 45, loop left] (cr5)
  (cr5) edge[-triangle 45] (cr4)
  (cr6a) edge[-triangle 45] (ec)
  (cr6b) edge[-triangle 45] (ec)
  (cr7) edge[-triangle 45] (cr6a)
  (cr7) edge[-triangle 45] (cr6b)
  (cr7) edge[>=triangle 45, loop right] (cr7)
  (cr8) edge[>=triangle 45, loop left] (cr8)
  (cr8) edge[-triangle 45] (cr4)
  (cr14) edge[-triangle 45, bend left = 20] (cr6a)
  (cr14) edge[-triangle 45] (ec)
  (ec) edge[>=triangle 45, loop right] (ec);
\end{tikzpicture}
\caption{Example \protect\ref{ExSpecial}}
\label{fig:ExSpecial}
\end{figure}
\end{example}

\subsection{Open questions}

We conclude with a short list of questions we have not been able to answer.

\begin{enumerate}
\item Does the equality $\tau (\mu ,q)=\min \{\tau _{L}(\mu ,q):L$ maximal
loop class$\}$ hold for all finite type measures $\mu $? If not, in what
generality does it hold?

\item In Example \ref{Pisot} we saw that the function $\tau _{E}(\mu ,q)$
can have a point of non-differentiability. Under what assumptions are the
functions $\tau _{L}(\mu ,q)$ differentiable for all $q$? real analytic?

\item What more can be learned about the multifractal analysis of $\mu$ from
the functions $\tau _{L}$ and their Legendre transforms?

\item The finite type model can be thought of as an analog of the Markov
chain model where the probabilities are replaced by matrices with
non-negative entries. Can one develop an analogous theory? For instance,
could one define an analogue of the notion of an invariant measure and use
this measure to obtain information about the self-similar measure?
\end{enumerate}

\appendix

\section{Details of Example \protect\ref{ExSpecial}}

\label{app}

As was noted in the example, there are four loop classes outside of the
essential class and the points associated with these loop classes have local
dimension bounded above from $\frac{\log 17/4 }{\log 3}$. Hence it suffices
to bound the local dimension of points in the essential class. One can check
that the transition matrices within the essential class are given by

\begin{align*}
T(\gamma_9,\gamma_{15}) & = \frac{1}{17} \left[ 
\begin{array}{ccc}
{4} & 0 & 0 \\ 
0 & {4} & {4} \\ 
{4} & 0 & 1%
\end{array}
\right] & T(\gamma_9,\gamma_{16}) & = \frac{1}{17} \left[ 
\begin{array}{cccc}
{\ 4 } & {\ 4 } & 0 & 0 \\ 
0 & 0 & {\ 4 } & {\ 4 } \\ 
0 & {\ 4 } & 0 & 1%
\end{array}
\right] \\
T(\gamma_9,\gamma_{17}) & = \frac{1}{17} \left[ 
\begin{array}{cccc}
0 & {\ 4 } & {\ 4 } & 0 \\ 
1 & 0 & 0 & {\ 4 } \\ 
{\ 4 } & 0 & {\ 4 } & 0%
\end{array}
\right] & T(\gamma_9,\gamma_{18}) & = \frac{1}{17} \left[ 
\begin{array}{ccc}
0 & {\ 4 } & {\ 4 } \\ 
1 & 0 & 0 \\ 
{\ 4 } & 0 & {\ 4 }%
\end{array}
\right] \\
T(\gamma_9,\gamma_{10}) & = \frac{1}{17} \left[ 
\begin{array}{ccc}
1 & 0 & {\ 4 } \\ 
{\ 4 } & 1 & 0 \\ 
0 & {\ 4 } & 0%
\end{array}
\right] & T(\gamma_9,\gamma_{11}) & = \frac{1}{17} \left[ 
\begin{array}{cc}
1 & 0 \\ 
{\ 4 } & 1 \\ 
0 & {\ 4 }%
\end{array}
\right] \\
T(\gamma_{10},\gamma_{12}) & = \frac{1}{17} \left[ 
\begin{array}{cc}
{\ 4 } & 0 \\ 
{\ 4 } & 1 \\ 
{\ 4 } & {\ 4 }%
\end{array}
\right] & T(\gamma_{10},\gamma_{13}) & = \frac{1}{17} \left[ 
\begin{array}{ccc}
{\ 4 } & {\ 4 } & 0 \\ 
0 & {\ 4 } & 1 \\ 
0 & {\ 4 } & {\ 4 }%
\end{array}
\right] \\
T(\gamma_{10},\gamma_9) & = \frac{1}{17} \left[ 
\begin{array}{ccc}
0 & {\ 4 } & {\ 4 } \\ 
{\ 4 } & 0 & {\ 4 } \\ 
0 & 0 & {\ 4 }%
\end{array}
\right] & T(\gamma_{11},\gamma_{10}) & = \frac{1}{17} \left[ 
\begin{array}{ccc}
1 & 0 & {\ 4 } \\ 
0 & {\ 4 } & 0%
\end{array}
\right] \\
T(\gamma_{11},\gamma_{11}) & = \frac{1}{17} \left[ 
\begin{array}{cc}
1 & 0 \\ 
0 & {\ 4 }%
\end{array}
\right] & T(\gamma_{12},\gamma_{12}) & = \frac{1}{17} \left[ 
\begin{array}{cc}
{\ 4 } & 0 \\ 
{\ 4 } & 1%
\end{array}
\right] \\
T(\gamma_{12},\gamma_{13}) & = \frac{1}{17} \left[ 
\begin{array}{ccc}
{\ 4 } & {\ 4 } & 0 \\ 
0 & {\ 4 } & 1%
\end{array}
\right] & T(\gamma_{12},\gamma_{9}) & = \frac{1}{17} \left[ 
\begin{array}{ccc}
0 & {\ 4 } & {\ 4 } \\ 
{\ 4 } & 0 & {\ 4 }%
\end{array}
\right] \\
T(\gamma_{13},\gamma_{10}) & = \frac{1}{17} \left[ 
\begin{array}{ccc}
{\ 4 } & 0 & 0 \\ 
1 & 0 & {\ 4 } \\ 
0 & {\ 4 } & 0%
\end{array}
\right] & T(\gamma_{13},\gamma_{13}) & = \frac{1}{17} \left[ 
\begin{array}{ccc}
{\ 4 } & {\ 4 } & 0 \\ 
0 & 1 & 0 \\ 
0 & 0 & {\ 4 }%
\end{array}
\right] \\
T(\gamma_{15},\gamma_{15}) & = \frac{1}{17} \left[ 
\begin{array}{ccc}
{\ 4 } & 0 & 0 \\ 
0 & {\ 4 } & {\ 4 } \\ 
{\ 4 } & 0 & 1%
\end{array}
\right] & T(\gamma_{15},\gamma_{16}) & = \frac{1}{17} \left[ 
\begin{array}{cccc}
{\ 4 } & {\ 4 } & 0 & 0 \\ 
0 & 0 & {\ 4 } & {\ 4 } \\ 
0 & {\ 4 } & 0 & 1%
\end{array}
\right] \\
T(\gamma_{15},\gamma_{17}) & = \frac{1}{17} \left[ 
\begin{array}{cccc}
0 & {\ 4 } & {\ 4 } & 0 \\ 
1 & 0 & 0 & {\ 4 } \\ 
{\ 4 } & 0 & {\ 4 } & 0%
\end{array}
\right] & T(\gamma_{15},\gamma_{18}) & = \frac{1}{17} \left[ 
\begin{array}{ccc}
0 & {\ 4 } & {\ 4 } \\ 
1 & 0 & 0 \\ 
{\ 4 } & 0 & {\ 4 }%
\end{array}
\right] \\
T(\gamma_{16},\gamma_{10}) & = \frac{1}{17} \left[ 
\begin{array}{ccc}
{\ 4 } & 0 & 0 \\ 
1 & 0 & {\ 4 } \\ 
{\ 4 } & 1 & 0 \\ 
0 & {\ 4 } & 0%
\end{array}
\right] & T(\gamma_{16},\gamma_{13}) & = \frac{1}{17} \left[ 
\begin{array}{ccc}
{\ 4 } & {\ 4 } & 0 \\ 
0 & 1 & 0 \\ 
0 & {\ 4 } & 1 \\ 
0 & 0 & {\ 4 }%
\end{array}
\right] \\
T(\gamma_{17},\gamma_{15}) & = \frac{1}{17} \left[ 
\begin{array}{ccc}
{\ 4 } & 0 & 0 \\ 
0 & {\ 4 } & {\ 4 } \\ 
{\ 4 } & 0 & 1 \\ 
{\ 4 } & 0 & {\ 4 }%
\end{array}
\right] & T(\gamma_{17},\gamma_{16}) & = \frac{1}{17} \left[ 
\begin{array}{cccc}
{\ 4 } & {\ 4 } & 0 & 0 \\ 
0 & 0 & {\ 4 } & {\ 4 } \\ 
0 & {\ 4 } & 0 & 1 \\ 
0 & {\ 4 } & 0 & {\ 4 }%
\end{array}
\right] \\
T(\gamma_{17},\gamma_{17}) & = \frac{1}{17} \left[ 
\begin{array}{cccc}
0 & {\ 4 } & {\ 4 } & 0 \\ 
1 & 0 & 0 & {\ 4 } \\ 
{\ 4 } & 0 & {\ 4 } & 0 \\ 
0 & 0 & {\ 4 } & 0%
\end{array}
\right] & T(\gamma_{17},\gamma_{18}) & = \frac{1}{17} \left[ 
\begin{array}{ccc}
0 & {\ 4 } & {\ 4 } \\ 
1 & 0 & 0 \\ 
{\ 4 } & 0 & {\ 4 } \\ 
0 & 0 & {\ 4 }%
\end{array}
\right] \\
T(\gamma_{18},\gamma_{10}) & = \frac{1}{17} \left[ 
\begin{array}{ccc}
1 & 0 & {\ 4 } \\ 
{\ 4 } & 1 & 0 \\ 
0 & {\ 4 } & 0%
\end{array}
\right] & T(\gamma_{18},\gamma_{11}) & = \frac{1}{17} \left[ 
\begin{array}{cc}
1 & 0 \\ 
{\ 4 } & 1 \\ 
0 & {\ 4 }%
\end{array}
\right] \\
& & 
\end{align*}

Consider the following following $K_{i}:=Cone(V_{i})=\{\sum
a_{j}v_{j}:v_{j}\in V_{i},\sum a_{j}\geq 1,a_{j}\geq 0\}$ where the $V_{i}$
are given by 
\begin{align*}
V_{9}& =\{[1,1,1],[4,0,0],[4,4,0],[8,0,8]\} \\
V_{10}& =\{[0,0,4],[0,4,0],[1,1,1],[2,0,2],[2,2,1/2],[16,4,4]\} \\
V_{11}& =\{[1,1],[2,0]\} \\
V_{12}& =\{[0,2],[1,1]\} \\
V_{13}& =\{[0,0,2],[0,4,3],[2,0,0],[2,8,0]\} \\
V_{15}& =\{[0,0,2],[1,1,1],[10,8,8],[20,0,52],[9/2,2,6]\} \\
V_{16}&
=\{[0,0,4,4],[1,1,1,1],[2,0,8,0],[10,8,8,0],[20,0,52,20],[9/2,2,6,2]\} \\
V_{17}&
=\{[0,4,4,0],[0,8,0,8],[0,52,20,16],[1,1,1,1],[2,6,2,5/2],[8,8,0,10]\} \\
V_{18}& =\{[1,1,1],[2,0,2],[4,4,0],[6,2,5/2],[8,0,10],[52,20,16]\}
\end{align*}%
One can check that $T(\gamma _{i},\gamma _{j})V_{i}\subset \frac{2}{17}V_{j}$
for all valid combinations of $\gamma _{i}$ and $\gamma _{j}$. This implies
that a product $\Vert T(\gamma _{i_{0}}\gamma _{i_{1}}\dots \gamma
_{i_{n}})\Vert \geq \left( \frac{2}{17}\right) ^{n}$ which proves the
desired result.

\end{document}